\documentclass[12pt,twoside,a4paper]{book}

\usepackage{CJKutf8}

\usepackage[T1]{fontenc}
\usepackage[brazil]{babel}
\usepackage[utf8]{inputenc}
\usepackage[pdftex]{graphicx}           
\usepackage{float}
\usepackage{setspace}                   
\usepackage{indentfirst}                
\usepackage{makeidx}                    
\usepackage[nottoc]{tocbibind}          
\usepackage{courier}                    
\usepackage{type1cm}                    
\usepackage{listings}                   
\usepackage{titletoc}
\usepackage[fixlanguage]{babelbib}
\usepackage[font=small,format=plain,labelfont=bf,up,textfont=it,up]{caption}
\usepackage[usenames,svgnames,dvipsnames]{xcolor}
\usepackage[a4paper,top=3cm,bottom=3.0cm,left=3cm,right=3cm]{geometry} 
\usepackage[pdftex,plainpages=false,pdfpagelabels,pagebackref,colorlinks=true,citecolor=DarkGreen,linkcolor=NavyBlue,urlcolor=DarkRed,filecolor=green,bookmarksopen=true]{hyperref} 
\usepackage[all]{hypcap}                    

\usepackage[square,sort,comma,numbers]{natbib}
\usepackage[obeyDraft]{todonotes}
\usepackage{amsmath}
\usepackage{amssymb}
\usepackage{amsthm}
\usepackage{eucal}
\usepackage{mathrsfs}

\renewcommand{\phi}{\varphi}

\renewcommand{\Re}{\operatorname{Re}}
\renewcommand{\Im}{\operatorname{Im}}

\newcommand{\N}{\mathbb{N}}
\newcommand{\C}{\mathbb{C}}
\newcommand{\R}{\mathbb{R}}

\newcommand{\D}{D}

\newcommand{\Qua}{\mathcal{Q}}

\newcommand{\bs}{\backslash}

\newcommand{\bbrac}[1]{\left\{ #1 \right\}}
\newcommand{\bsbrac}[1]{\left[ #1 \right]}
\newcommand{\bpar}[1]{\left( #1 \right)}

\newcommand{\babs}[1]{\left| #1 \right|}
\newcommand{\ddint}[4]{\displaystyle\int_{#1}^{#2} \! #3 \, \operatorname{d} \! #4} 

\newcommand{\Ho}{\mathcal{O}}
\newcommand{\Hos}[1]{\Ho_{{(#1)}}}
\newcommand{\Cs}[1]{\C_{(#1)}}

    {\@beginparpenalty\predisplaypenalty
     \@endparpenalty\postdisplaypenalty
     \refstepcounter{equation}%
     \trivlist \item[]\leavevmode
       \hb@xt@\linewidth\bgroup $\m@th
         \displaystyle
         \hskip\mymathindent}%
        {$\hfil 
         \displaywidth\linewidth\hbox{\@eqnnum}%
       \egroup
     \endtrivlist}

\newcommand{\fs}{\hat}

\newcommand{\dint}{\ \!\mathrm{d}}
\newcommand{\sen}{\operatorname{sen}}
\newcommand{\Disc}{\textrm{D}}
\newcommand{\del}{\partial}

\newcommand{\gapprox	}[1]{\approx_{#1}} 
\newcommand{\tays}[1]{\mathcal{T}_{(#1)}} 
\newcommand{\B}[1]{\mathcal{B}_{#1}} 
\newcommand{\Bo}[1]{\mathcal{B}_{#1}} 
\newcommand{\fBo}[1]{\fs{\mathcal{B}}_{#1}}
\newcommand{\La}{\mathcal{L}} 
\newcommand{\fLa}{\fs{\La}} 

\newcommand{\Sw}{\mathcal{S}}
\newcommand{\F}{\mathcal{F}}

\newcommand{\ssubset}{\subset \!\!\! \subset}

\newtheorem{dfn}{Definição}
\newtheorem{exe}{Exemplo}
\newtheorem{obs}{Observação}
\newtheorem{lem}{Lema}
\newtheorem{prp}{Proposição}
\newtheorem{teo}{Teorema}
\newtheorem{cor}{Corolário}

\fontsize{60}{62}\usefont{OT1}{cmr}{m}{n}{\selectfont}

\usepackage{fancyhdr}
\graphicspath{{./imagens/}}             
\makeindex                              
\fontsize{60}{62}\usefont{OT1}{cmr}{m}{n}{\selectfont}
\cleardoublepage
\normalsize

\begin{document}
\frontmatter 
\fancyhead[RO]{{\footnotesize\rightmark}\hspace{2em}\thepage}
\setcounter{tocdepth}{2}
\fancyhead[LE]{\thepage\hspace{2em}\footnotesize{\leftmark}}
\fancyhead[RE,LO]{}
\fancyhead[RO]{{\footnotesize\rightmark}\hspace{2em}\thepage}

\onehalfspacing  

\thispagestyle{empty}
\begin{center}
    \vspace*{2.3cm}
    \textbf{\Large{A Equação de Euler e a\\
     Análise Assintótica de Gevrey}}\\
    
    \vspace*{1.2cm}
    \Large{Max Reinhold Jahnke}
    
    \vskip 2cm
    \textsc{
    Dissertação apresentada\\[-0.25cm] 
    ao\\[-0.25cm]
    Instituto de Matemática e Estatística\\[-0.25cm]
    da\\[-0.25cm]
    Universidade de São Paulo\\[-0.25cm]
    para\\[-0.25cm]
    obtenção do título\\[-0.25cm]
    de\\[-0.25cm]
    Mestre em Ciências}
    
    \vskip 1.5cm
    Programa: Matemática Aplicada\\
    Orientador: Prof. Dr. Paulo Domingos Cordaro\\

   	\vskip 1cm
    \normalsize{Durante o desenvolvimento deste trabalho o autor recebeu auxílio financeiro da CNPq.}
    
    \vskip 0.5cm
    \normalsize{São Paulo, julho de 2013}
\end{center}

%
%
%
%

%
%

%
%
%
%

\newpage
\mbox{}
\newpage
\thispagestyle{empty}
    \begin{center}
        \vspace*{2.3 cm}
        \textbf{\Large{A Equação de Euler e a\\
     Análise Assintótica de Gevrey}}\\
        \vspace*{2 cm}
    \end{center}
    \vskip 2cm
    \begin{flushright}
	Esta versão da dissertação contém as correções e alterações sugeridas\\
	pela Comissão Julgadora durante a defesa da versão original do trabalho,\\
	realizada em 04/10/2013. Uma cópia da versão original está disponível no\\
	Instituto de Matemática e Estatística da Universidade de São Paulo.

    \vskip 2cm

    \end{flushright}
    \vskip 4.2cm

    \begin{quote}
    \noindent Comissão Julgadora:
    
    \begin{itemize}
		\item Prof. Dr. Paulo Domingos Cordaro (orientador) - IME-USP
		\item Prof. Dr. Clodoaldo Grotta Ragazzo - IME-USP
		\item Prof. Dr. Gerson Petronilho - UFSCar
    \end{itemize}
      
    \end{quote}
\pagebreak

\newpage
\chapter*{}
\null
\vfill
	\begin{flushright}
    	 \begin{CJK}{UTF8}{min}
「色は匂へど　\\散りぬるを　\\
　我が世誰ぞ　\\常ならむ　　\\
　有為の奥山　\\今日越えて　\\
　浅き夢見じ　\\酔ひもせず」\\
		 \end{CJK}
\vskip 0.2cm
Autor desconhecido.
\vskip 1cm

\textit{``Although its scent still lingers on\\
the form of a flower has scattered away\\
For whom will the glory\\
of this world remain unchanged?\\
Arriving today at the yonder side\\
of the deep mountains of evanescent existence\\
We shall never allow ourselves to drift away\\
intoxicated, in the world of shallow dreams.''}
\vskip 0.2cm
Tradução para o inglês feita pelo professor Ryuichi Abe.
    \end{flushright}

\newpage
\mbox{}
\newpage

\pagenumbering{roman}     
\thispagestyle{empty}

\chapter*{Agradecimentos}

Agradeço primeiro ao Prof. Dr. Paulo Domingos Cordaro pela orientação durante toda a escrita deste trabalho, pelo constante exemplo de profissionalismo e pela contagiante paixão pela matemática.

Aos membros da banca examinadora, o Prof. Dr. Clodoaldo Grotta Ragazzo e o Prof. Dr. Gerson Petronilho, agradeço por terem analisado cuidadosamente o texto e dado valiosas sugestões.

A Adèle Helena Ribeiro eu agradeço por ter lido cuidadosamente cada linha, tornando o texto muito mais agradável de ler, e por ter cuidado de mim e ser muito paciente durante fases difíceis do mestrado.

Deixo também um agradecimento especial aos professores Sônia Regina Leite Garcia e Manuel Valentim de Pera Garcia que, apesar de não terem participado diretamente no meu mestrado, foram importantes na minha formação como matemático.

Também sou grato aos meus amigos Luis Fernando Ragognette e Gabriel C. C. S. de Araújo por me ajudarem com a notação e estilo, e por darem diversas dicas valiosas. Também por terem pacientemente assistido às exposições deste texto. 

Agradeço ao Nicholas Braun Rodrigues por ter lido uma versão inicial da dissertação e dado valiosas sugestões para deixar o texto mais compreensível.

Não posso deixar de reconhecer a importância dos amigos que estiveram presentes durante quase toda a minha graduação e mestrado.

Agradeço ao Pedro Henrique Pontes por não esquecer de voltar para o Brasil e sempre convidar os amigos para jogar videogame e comer pizza.

Ao Bruno de Paula Jacóia e à Priscila Freitas, agradeço por terem organizado diversos encontros gastronômicos e pelos vários jogos de tabuleiro, que foram importantíssimos para eu manter minha sanidade.

Também preciso agradecer ao Lucas Ruiz dos Santos pela simpatia contagiante e pelo infinito bom humor.

Apesar de ter fugido para Alemanha e ainda não ter visitado os amigos, agradeço ao Gabriel Zanetti Nunes Fernandes por ter me dado apoio em diversas situações importantes, e por, mesmo agora morando tão longe, conseguir estar presente.

Não posso deixar de citar minha família. Meus irmãos Viktor Jahnke e Cristiane Jahnke, que apesar de serem mais novos, sempre me serviram de exemplo. Aos meus pais, Gilda Timóteo Leite e Horst Reinhold Jahnke, que sempre se preocuparam com minha educação e me deram as melhores condições que puderam.

Agradeço a todos que participaram do ZFC Fan Club e ajudaram a tornar o IME ainda mais divertido e acolhedor.

Agradeço também ao CNPq pelo financiamento, que me permite dedicação integral ao estudo da Matemática.

\chapter*{Resumo}

\noindent Jahnke, M. R. \textbf{A Equação de Euler e a Análise Assintótica de Gevrey}. 
2013. 56 f.
Dissertação (Mestrado) - Instituto de Matemática e Estatística,
Universidade de São Paulo, São Paulo, 2013.
\\


Neste trabalho, introduzimos a noção de desenvolvimento assintótico em classes de Gevrey e mostramos como o conceito clássico de convergência de séries de potências pode ser generalizado para englobar o caso em que o raio de convergência é nulo. Essa técnica pode ser útil em situações em que é necessário trabalhar com séries formais, como no estudo de Equações Diferenciais.

Caracterizamos o conjunto das funções holomorfas que admitem desenvolvimento assintótico e, em cada classe de Gevrey, definimos uma aplicação que associa uma função a uma série formal.

Determinamos sob quais condições tal aplicação é sobrejetora e sob quais ela é injetora, possibilitando a ampliação do conceito de convergência e as aplicações da teoria. 

Além disso, mostramos como essa técnica pode ser usada para obter resultados em equações diferenciais. Para isso, fazemos uma breve introdução de Equações Diferenciais com uma variável complexa e introduzimos o conceito de Polígono de Newton, ferramenta que permite obter a classe de Gevrey de uma solução formal.

Finalmente, encontramos condições para que a soma de uma solução formal de uma equação diferencial seja uma solução clássica.
\\

\noindent \textbf{Palavras-chave:} Equação de Euler, Desenvolvimento Assintótico, Classes de Gevrey.

\chapter*{Abstract}
\noindent Jahnke, M. R. \textbf{Euler Equation and Gevrey Asymptotic Analysis}. 
2013. 56 f.
Dissertação (Mestrado) - Instituto de Matemática e Estatística,
Universidade de São Paulo, São Paulo, 2010.
\\

In this work, we introduce the notion of Gevrey asymptotic expansion and we show how the classical concept of a convergent power series can be generalized to include the case in which the radius of convergence is zero. This technique can be useful in situations where it is necessary to work with formal power series, as in the study of Differential Equations.

We characterize the set of holomorphic functions which admit Gevrey asymptotic expansion and we define in each Gevrey class a map that associates to function in the class a formal series.

We determine under which conditions such a map is surjective and under which it is injective, allowing the extension of the concept of convergence and applications of the theory.

Furthermore, we show how this technique can be used to obtain results in Differential Equations. For this, we briefly recall the theory of Differential Equations in one complex variable and we introduce the concept of the Newton Polygon, a tool that allows us to find the Gevrey class of a formal solution.

Finally, we find suficient conditions for the sum of a formal solution of a differential equation to be a classical solution.

\noindent \textbf{Keywords:} Euler's Equation, Asymptotic Development, Gevrey classes.


\tableofcontents    






\mainmatter

\fancyhead[RE,LO]{\thesection}

\singlespacing              

\chapter{Introdução}
\label{cap:introducao}

\section{Breve revisão histórica}

Em diversas áreas da Matemática, como Equações Diferenciais e Combinatória, é comum assumir que a solução de um determinado problema é dada na forma de uma série de potências. Em geral, é feita uma manipulação formal da série para determinar uma expressão para os coeficientes.

Infelizmente, muitas vezes a série resultante não é convergente no sentido clássico, isto é, o raio de convergência é zero. Existem diversas técnicas para tratar esse caso. Um dos trabalhos mais antigos foi feito por Leibniz, que atribuiu à série \(\sum_{k=0}^\infty (-1)^k\) o valor \(1/2\).

Leibniz chegou a essa conclusão usando diversos argumentos, uns mais rigorosos que outros. Um deles, mais próximo dos nossos padrões de rigor, é o que descreveremos a seguir.

Para \(|x| < 1\), a função \(1/(1+x)\) pode ser representada pela série de potências
\[
	\frac{1}{1+x} = 1 - x + x^2 - x^3 + \ldots.
\]

Como o lado esquerdo da expressão acima está definido para \(x = 1\), Leibniz argumentou que pela ``lei da continuidade'' tem que valer
\[
	\sum_{k=0}^\infty (-1)^k = \frac{1}{2}.
\]

Essa abordagem foi formalizada e hoje é conhecida como ``soma de Abel''. Se \(\sum_{k=0}^\infty a_k\) é uma série tal que \(\sum_{k=0}^\infty a_kz^k\) converge em um disco unitário, então podemos dizer que essa série é Abel somável e definir sua soma como:
\[
	\sum_{k=0}^\infty a_k \doteq \lim_{x \to 1^- \atop{0 < x < 1}} \sum_{k=0}^\infty a_kx^k.
\]

Para uma grande classe de séries divergentes, esse método surge naturalmente e é bem adequado para tratar diversos problemas em Análise. Por exemplo, foi usando técnicas assim que Euler foi capaz de descobrir a função Zeta de Riemann cem anos antes de Riemann.

Porém, esse método falha para séries cujas respectivas séries de potência possuem raio de convergência zero, como a série fatorial
\[
	\sum_{k=0}^\infty (-1)^kk!.
\]

Euler, em \cite{euler1746}, chamou essas séries de \textit{séries divergentes por excelência} e atribuiu um valor à série fatorial, o que gerou muita polêmica na época.

As ideias de Euler para tratar esse tipo de problema foram geniais e deram frutos em várias áreas de pesquisa.

Uma das ideias utilizadas foi transformar o problema de somar a série em um problema de Equações Diferenciais. Euler transformou a série acima em uma série formal,
\[
	\fs{f}(z) = \sum_{k=0}^\infty (-1)^kk!z^{k+1},
\]
e observou que essa série satisfaz formalmente uma equação diferencial:
\[
 	z^2 \fs{f}' + \fs{f} = z.
\]

Hoje, a série é conhecida como \textit{Série de Euler} e a equação diferencial como \textit{Equação de Euler}.

As ideias de Euler começaram a ser melhor entendidas depois do artigo de Borel \cite{borel1899memoire}, onde foi desenvolvida uma das ferramentas essenciais para somar séries divergentes: a transformada de Borel. Esse trabalho também apresentou diversas aplicações em Equações Diferenciais, entre elas a Equação de Euler.

Essas ideias foram combinadas com a teoria de Expansão Assintótica clássica. Essa teoria foi desenvolvida por Poincaré em \cite{poincare1993memoire} para trabalhar com soluções formais de Equações Diferenciais Analíticas e relacioná-las com funções que são soluções clássicas dessas equações.

No entanto, essa teoria tem um grande problema. Dada uma série formal, não existe uma única função clássica relacionada a essa série. Para lidar com esse problema, Watson, em \cite{watson1912transformation}, e Nevanlinna, \cite{nevanlinna1921theorie}, introduziram o conceito de Expansão Assintótica de Gevrey.

Essa teoria ficou esquecida até que foi reintroduzida por Ramis, em \cite{ramis1978devissage}, que desenvolveu sistematicamente a Análise Assintótica de Gevrey com Equações Diferenciais Ordinárias no domínio complexo.

\section{Apresentação deste trabalho}

No primeiro capítulo deste trabalho, faremos uma reinterpretação da teoria de Desenvolvimento Assintótico de Gevrey apresentada por Malgrange em \cite{malgrange1995sommation} e por Balser em \cite{balser1994divergent}.

A notação foi modificada e alguns resultados foram ligeiramente adaptados visando unificar os conceitos de ambas as fontes e simplificar a aplicação em Equações Diferenciais.

Introduziremos os conceitos de séries de Gevrey, relacionamos essas séries com determinadas classes de funções holomorfas e finalmente introduzimos o conceito de Análise Assintótica de Gevrey.

As séries de Gevrey são séries formais cujos coeficientes não crescem muito rapidamente. Nesse caso, é possível estender a noção de convergência e encontrar uma soma, que é uma função holomorfa definida em um tipo especial de domínio.

Estudamos sob quais condições há uma única soma de uma dada série divergente. Para encontrar tais condições, definimos uma aplicação especial, chamada aplicação de Taylor, que relaciona séries de Gevrey a uma classe de funções.

Determinamos sob quais condições a aplicação de Taylor é injetora e sob quais condições ela é sobrejetora. Tais resultados são obtidos usando o Lema de Watson, uma das ferramentas mais importantes desta teoria e uma caracterização das séries formais somáveis, obtida no mais importante teorema deste trabalho, o Teorema \ref{caracdasoma} da página \pageref{caracdasoma}.

No segundo capítulo, apresentaremos aplicações em Equações Diferenciais Ordinárias. Antes disso, faremos uma breve introdução da teoria de Equações Diferenciais Ordinárias em uma variável complexa.

Introduziremos uma ferramenta chamada Polígono de Newton, que nos ajuda a determinar a classe de Gevrey de soluções formais de uma dada equação diferencial e assim poder aplicar a teoria desenvolvida no primeiro capítulo.

Finalmente, provaremos o teorema que nos fornece condições para que uma série formal de uma equação diferencial seja somável e que seja uma solução clássica do problema estudado.

\section{Outras aplicações}

Além de aplicações em Equações Diferenciais Ordinárias, essas técnicas de somabilidade de séries divergentes também tem sido aplicadas em Equações Diferenciais Parciais.

Um dos primeiros trabalhos nessa direção foi feito com aplicações na Equação do Calor, publicado por Lutz, D. A. and Miyake, M. and Sch{\"a}fke, R em \cite{lutz1999borel}. Esse trabalho foi estendido no artigo de Balser, W. e Loday-Richaud, M. \cite{balser2009summability}.

Aplicações em casos mais gerais foram desenvolvidas no artigo de Balser \cite{balser2004summability} e de \=Ouchi, S. 
\cite{ouchi2002multisummability}.




\chapter{Expansão assintótica de Gevrey}

\section{Motivação}

Em diversas áreas da Matemática, como Equações Diferenciais e Combinatória, é comum assumir que a solução de um determinado problema é dada na forma de uma série de potências e fazer uma manipulação formal da série para determinar uma expressão para os coeficientes.

Para exemplificar o uso de séries formais, consideremos o seguinte problema trivial:

\begin{exe}[Função exponencial] Queremos encontrar uma função holomorfa \(f\), definida em uma vizinhança de \(0\), que satisfaz o seguinte problema de Cauchy:
	\[
		\left\{
			\begin{array}{  ll}
				f' & = f; \\
				f(0) & = 1.
			\end{array}
		\right.
	\]
	
	Supondo que \(f\) pode ser escrita como série de potências
	\(
		f(z) = \sum_{k=0}^\infty a_k z^k,
	\)
	derivando formalmente e substituindo na equação diferencial, temos
	\[
		f'(z) = \sum_{k=1}^\infty ka_k z^{k-1} = f(z) = \sum_{k=0}^\infty a_k z^k,
	\]
	que nos dá a relação entre os coeficientes da série:
	\(
		ka_k = a_{k-1}, \quad k \geq 1.
	\)

	Da condição inicial, segue que \(a_0 = 1\). Portanto, \(a_1 = 1,~a_2 = 1/2, ...,~a_k = 1/k!\). Assim, temos uma expressão explícita da série de \(f\),
	\[
		f(z) = \sum_{k=0}^\infty \frac{z^k}{k!},
	\]
	que converge uniforme e absolutamente em qualquer compacto de \(\C\). Logo \(f\) é uma função holomorfa definida em \(\C\).
\end{exe}

A técnica apresentada acima funciona em diversas situações, mas infelizmente possui limitações. Considere, por exemplo, o seguinte problema de Cauchy.

\begin{exe}[Equação de Euler]
\label{des:equacao_de_euler}
Queremos encontrar uma solução \(f\), derivável numa vizinhança de \(0\), que satisfaz a seguinte equação diferencial:
	\[
		\left\{
			\begin{array}{ll}
				z^2f' + f & = z; \\
				f(0) & = 0.
			\end{array}
		\right.
	\]
	
	Supondo \(f(z) = \sum_{k=0}^\infty a_kz^k\) e procedendo de maneira análoga ao exemplo anterior, obtemos
	\[
	  \begin{aligned}
		  z^2f'(z) + f(z) & = \sum_{k=1}^\infty ka_kz^{k+1} + \sum_{k=0}^\infty a_kz^k \\
		    & = a_0 + a_1z + \sum_{k=2}^\infty \left[(k-1)a_{k-1} + a_k\right]z^k \\
		    & = z.
	  \end{aligned}
	\]

	Logo \(a_0 = 0\), \(a_1 = 1\) e, para \(k \geq 2\), \(a_k = - (k-1)a_{k-1}\). Portanto \(a_k = (-1)^{k-1}(k-1)!\) para \(k \geq 1\) e \(f\) pode ser escrita como
	\[
		f(z) = \sum_{k = 1}^\infty (-1)^{k-1}(k-1)! z^k.
	\] 
	Essa série é conhecida como Série de Euler. Observe que para \(z \neq 0\) o termo geral da série não vai a zero e portanto não é convergente em nenhum disco centrado na origem. Logo, usando o conceito clássico de convergência de séries, não é possível usar a série para definir a função \(f\) que procuramos.
\end{exe}

Queremos expandir nosso conceito de convergência de séries de potências de forma a obter uma ``soma'' das soluções formais, como a apresentada anteriormente, que pode ser usada para entender melhor o problema estudado.

Mais precisamente, dada uma série formal \(\fs{f}(z) = \sum_{k=0}^\infty a_k z^k/k!\), queremos encontrar uma função \(f\) que possui \(\fs{f}\) como ``série de Taylor''. Como a série não converge absolutamente, não conseguiremos definir a função holomorficamente em uma vizinhança da origem, então precisamos dizer em que tipo de domínio a função será definida e em qual sentido \(\fs f\) será a série de Taylor desta função.

Além disso, pode não existir uma única função \(f\) que tem \(\fs{f}\) como série de Taylor. Para ter algum controle sobre essa situação, precisamos colocar algumas restrições sobre como os termos da série \(\fs{f}\) crescem.

\section{Séries formais de Gevrey}

Antes de dizer que tipo de controle precisamos ter das séries que trabalharemos, vamos motivar nossa próxima definição com um exemplo.

\begin{exe}[Estimativa de Cauchy]
\label{des:serie_absconv} Suponha que a série \(\sum_{k=0}^\infty a_k z^k\) converge absolutamente no disco \(\overline{D}_R(0)\), com \(R > 0\). Essa série define uma função \(f\) holomorfa em \(D_R(0)\), com \(a_k = f^{k}(0)/k!\), para \(k = 0, 1, 2, ...\).

Se \(0 < r < R\) e \(z \in D_R(0)\) for tal que \(D_r(z) \subset D_R(0)\), podemos usar a Fórmula Integral de Cauchy para obter
\[
	f^{(k)}(z) = \frac{k!}{2\pi i} \int_{|\xi - z| = r} \frac{f(\xi)}{(\xi - z)^{k+1}} \dint \xi.
\]
E vale a seguinte estimativa, conhecida como Estimativa de Cauchy:
\[
|f^{(k)}(z)| \leq k! \sup_{|\xi - z| = r} |f(\xi)| \left(\frac{1}{r}\right)^k.
\]

Em particular, existem constantes \(M\) e \(C\) positivas tais que
\begin{equation}
	\label{des:estimativa_cauchy}
	|a_k| \leq MC^k, \quad k \geq 0.
\end{equation}
\end{exe} 

Queremos um controle parecido com o acima, mas permitindo que os termos \(a_k\) cresçam um pouco mais rapidamente. Isso nos motiva a seguinte definição.

\begin{dfn} Seja \(s \geq 1 \). Uma série formal \(\sum_{k=0}^\infty a_k z^k\) é dita \textbf{série de Gevrey de ordem s} se existem constantes \(M,C > 0\) tais que
\begin{equation}
\label{des:serie_gevrey}
|a_k| \leq MC^k k!^{s-1}, \quad k \geq 0.
\end{equation}
Por brevidade, muitas vezes apenas diremos que a série é \textbf{Gevrey-s}. Denotamos por \(\C_{(s)}[[z]]\) o conjunto das séries formais de Gevrey de ordem s.
\end{dfn}

Já possuímos alguns exemplos de séries de Gevrey. A série de Euler, 
que obtivemos no exemplo \eqref{des:equacao_de_euler}, é Gevrey de ordem 2. Vimos no exemplo \eqref{des:serie_absconv} que toda série absolutamente convergente é Gevrey de ordem 1.

\begin{obs}
\label{gevrey1holomorfa}
A recíproca da última afirmação também é válida, ou seja, se \(\fs{f}(z) = \sum_{k=0}^\infty a_k z^k\) é Gevrey de ordem 1, então existem constantes positivas \(M,C\) tais que vale \eqref{des:serie_gevrey}. Se \(|z| < 1/C\), então vale
\[
	\sum_{k=0}^\infty \left|a_k\right| |z|^k \leq \sum_{k=0}^\infty M(C|z|)^k \leq M\frac{1}{1 - C|z|} < \infty.	
\]

Logo a série \(\fs{f}\) é convergente no disco \(D_{1/C}(0)\). Provamos que \(\C_{(1)}[[z]] = \C\{z\}\), o conjunto das séries de potências que possuem raio de convergência positivo.
\end{obs}

Nosso próximo resultado nos indica algumas operações que podemos fazer com séries formais sem mudar a ordem das séries.

\begin{prp} O conjunto \(\C_{(s)}[[z]]\) é uma álgebra diferencial.
\end{prp}

\begin{proof} Sejam \(\fs{f}(z) = \sum_{k=0}^\infty a_kz^k\) e \(\fs{g} = \sum_{k=0}^\infty b_kz^k \in \C_{(s)}[[z]]\). Sejam \(M,C > 0\) tais que a desigualdade \eqref{des:serie_gevrey} vale para \(a_k\) e \(b_k\). É trivial provar que, se \(\lambda \in \C\), então \(\lambda \fs{f} \in \C_{(s)}[[z]]\). Também segue direto da Desigualdade Triangular que \(\fs{f} + \fs{g} \in \C_{(s)}[[z]]\).

Vamos provar que \(\fs{f}\fs{g} \in \C_{(s)}[[z]]\). Temos que \(\fs{f}\fs{g} = \sum_{k=0}^\infty \left( \sum_{l=0}^k a_l b_{k-l}\right) z^k\), portanto
	\[
		\begin{aligned}
			\left| \sum_{l=0}^k a_l b_{k-l}\right|
				& \leq \sum_{l=0}^k |a_l| |b_{k-l}| \\
				& \leq \sum_{l=0}^k MC^ll!^{s-1} MC^{k-l}(k-l)!^{s-1} \\
				& = M^2C^kk!^{s-1} \sum_{l=0}^k \binom{k}{l}^{1-s}
		\end{aligned} 
	\]
	
	Usando que as desigualdades \(\binom{k}{l}^{1-s} \leq 1\) e \(k \leq 2^k\) valem para \(k \geq 1\) e \(l \leq k\), temos que 
	\[
		\left| \sum_{l=0}^k a_l b_{k-l}\right| \leq 2M^2 (2C)^kk!^{s-1}.
	\]
	
	Logo \(\fs{f}\fs{g} \in \C_{(s)}[[z]]\). 
	
	Falta verificarmos que \(\fs{f}'(z) = \sum_{k=1}^\infty ka_kz^{k-1} = \sum_{k=0}^\infty (k+1)a_{k+1}z^k \doteq \sum_{k=0}^\infty c_kz^k\) é Gevrey-\(s\).
	Usando que \(k \leq 2^{k-1}\) e \(C^k \leq C'^{k-1}\), onde \(C' = \sup_{k \geq 2} C^{\frac{1}{k-1}}C\), temos
	\[
	\begin{aligned}
		|c_k| 
			& \leq k MC^kk!^{s-1} \\
			& \leq 2^{k-1}MC'^{k-1}(k-1)!^{s-1}(2^{k-1})^{s-1} \\
			& = M(2^{s}C')^{k-1}(k-1)!^{s-1}.
	\end{aligned}
	\]
	Como queríamos.
\end{proof}

\section{Setores do plano complexo}

Seja \(\fs{f}\) uma série de Gevrey de ordem \(s\). Como vimos, no caso \(s = 1\), \(\fs{f}\) define uma função holomorfa em uma vizinhança da origem, logo pode ser tratado com a teoria clássica de Análise Complexa. Nos focaremos no caso \(s > 1\).

Nesse caso, como a série não converge, não conseguiremos definir uma função holomorfa em uma vizinhança da origem que possui \(\fs{f}\) como série de Taylor. Mas conseguimos definir uma função holomorfa \(f\) em um conjunto aberto do plano que possui 0 como ponto de acumulação e, num sentido que ficará mais claro adiante, a função \(f\) terá \(\fs{f}\) como ``série de Taylor''. A seguir definimos em que tipo de domínio a função \(f\) será definida.

\begin{dfn}
	Um \textbf{setor (aberto)} do plano complexo é um conjunto da forma
	\[ S = \{ z \in \C: 0 < |z| < r, \eta < \arg z < \theta \} \]
	com \(-\pi \leq \eta < \theta \leq \pi\) e \(r > 0 \).
\end{dfn}
\todo{reler esse parágrafo com calma}
Dizemos que \(\alpha = \theta - \eta\) é a abertura do setor e que \(\psi = \eta + \alpha/2\) é a direção do setor \(S\). Também usamos a notação \(S(\psi, \alpha, r)\) para representar o setor \[S = \{ z \in \C: 0 < |z| < r, -\alpha/2 < \arg z - \psi < \alpha/2 \}\] e \(S(\psi, \alpha)\) para representar o setor de raio infinito.

Se \(S'\) é outro setor, dado por \(r',\eta',\theta'\), com \( 0 < r'< r \) e \( \eta < \eta' < \theta' < \theta\), dizemos que \(S'\) está \textbf{estritamente contido} em \(S\) e escrevemos \(S' \ssubset S\).

O seguinte lema é um resultado muito simples que nos possibilita construir subconjuntos de setores que facilitam a aplicação de ferramentas úteis, como a Fórmula Integral de Cauchy, o Teorema da Desigualdade do Valor Médio e a obtenção de algumas estimativas que vamos precisar.

\begin{lem}
	Dado \(S' \ssubset S\), existe \(\delta > 0\) tal que o disco \(D(z, |z|\delta)\) está estritamente contido em \(S\) para todo \(z \in S'\).
\end{lem}

\begin{proof}
Seja \(0 < t < \min\{r', 1\}\). Como \(\overline{\{z \in S': |z| = t\}}\) é compacto e está contido em \(S\), existe \(\delta > 0\) tal que, para todo \(z\) em \(S'\), com \(|z| = t,\ \overline{\Disc}(z,\delta) \subset S\). Diminuindo \(\delta\), se necessário, podemos supor que \(r'(\delta + 1) < r\).

Seja \(w \in \overline{\Disc}(z, |z|\delta)\). Vamos mostrar que \(w \in S\).

Podemos escrever \(w = z + (w - z)\) e, multiplicando ambos os lados por \(t/|z|\), obtemos
\[
	\frac{tw}{|z|} = \frac{tz}{|z|} + \frac{t(w - z)}{|z|}.
\]

Vamos primeiro ver qual é o argumento de \(w\). Como \(|\frac{t(w-z)}{z}| < t\delta\), segue que \(\frac{tw}{|z|} \in \Disc(\frac{tz}{|z|},t\delta) \subset S\). Logo, \(\eta < \arg(\frac{tw}{|z|}) < \theta \) e portanto \(\eta < \arg(w) < \theta \).

Agora vamos estimar o módulo de \(w\). Temos que \(|w| = |z + (w - z)| \leq |z| + |z|\delta = r'(1 + \delta) < r\).

Concluímos que \(w \in S\).
\end{proof}


Trabalharemos com funções cujas derivadas próximas da origem podem crescer mais rapidamente que as estimativas de Cauchy e, portanto, não são holomorfas em uma vizinhança de 0. Porém ainda precisamos de um certo controle do quão rápido as derivadas podem crescer. Elas precisam crescer obedecendo estimativas de Gevrey.

\begin{dfn} Para \(s \geq 1\), definimos \(\Hos{s}(S)\) como o subespaço de \(\Ho(S)\) formado pelas \(f\) holomorfas em \(S\) que satisfazem a seguinte propriedade: para todo \(S' \ssubset S\), existem constantes \(M,C > 0\), que podem depender de \(S'\), tais que
\begin{equation}
\label{des:funcao_gevrey}
	|f^{(k)}(z)| \leq MC^k k!^s, \quad k \geq 0,~ z \in S'.
\end{equation}

\end{dfn}

Facilmente podemos verificar que, se \(s < s'\), então \(\Hos{s}(S) \subset \Hos{s'}(S)\). A próxima proposição nos dá algumas desigualdades equivalentes àquela vista em \eqref{des:funcao_gevrey}. Essas novas estimativas simplificarão alguns cálculos que veremos no decorrer deste trabalho.

\begin{prp} A estimativa \eqref{des:funcao_gevrey} é equivalente às seguintes estimativas:
	\begin{equation}
		\label{des:funcao_gevrey1}
		|f^{(k)}(z)| \leq MC^k k^{ks}, \quad k \geq 0,~ z \in S';
	\end{equation}
	\begin{equation}
		\label{des:funcao_gevrey2}
		|f^{(k)}(z)| \leq MC^k \Gamma(1 + ks), \quad k \geq 0,~ z \in S'.
	\end{equation}
 Em cada estimativa, \(C\) e \(M\) representam constantes adequadas. 
\end{prp}

\begin{proof}  Da desigualdade \(n! \leq n^n\) segue direto que \eqref{des:funcao_gevrey} implica \eqref{des:funcao_gevrey1}. Por outro lado, como \(e^n \geq n^n/n!\), temos que \(n^n \leq e^nn!\) e segue facilmente que \eqref{des:funcao_gevrey1} implica \eqref{des:funcao_gevrey}.

Da fórmula de Stirling, sabemos que, para \(t > 0\), vale \( t^t \leq e^t\Gamma(1 + t) \) e, \(s^nn^{sn} \leq e^{sn}\Gamma(1 + ns)\). Usando essa desigualdade temos que \eqref{des:funcao_gevrey1} implica \eqref{des:funcao_gevrey2}. Finalmente, usando novamente a fórmula de Stirling, obtemos, para \(t\) grande, que \(\Gamma(1 + t) \leq t^t\). Isso é suficiente para provarmos que \eqref{des:funcao_gevrey2} implica \eqref{des:funcao_gevrey1}.
\end{proof}

Vamos prosseguir com um lema que nos permitirá definir a noção de série de Taylor, expandida a partir da origem, de funções definidas em setores. 

\begin{lem}
\label{cap:des:limorigem} Se \(f\in \Hos{s}(S)\), então, para todo setor \(S' \ssubset S\), os limites
	\begin{equation}
	\label{lem:termo_taylor}
		a_k = \lim_{z \to 0 \atop{z \in S'}} f^{(k)}(z), \quad k \geq 0,
	\end{equation}
existem e são independentes de \(S'\).
\end{lem}

\begin{proof} Para provarmos a existência do limite, vamos tomar arbitrariamente uma sequência \(\{z_n\}\) do subsetor \(S' \ssubset S\) que converge para 0 e mostrar que a sequência \(\{f(z_n)\}\) é de Cauchy.

Para cada par \(m,n \in \N\), a Desigualdade do Valor Médio nos garante que vale a seguinte estimativa: 
\[
	\begin{aligned}
		|f^{(k)}(z_n) - f^{(k)}(z_m)  | & \leq \sup_{S'} \left|f^{(k+1)}\right||z_n - z_m| \\
	 & \leq MC^{k+1}(k+1)!^s|z_n - z_m|.
	\end{aligned}
\]

Segue que a sequência \(\{f(z_n)\}\) é de Cauchy.

Para verificarmos que o limite não depende da sequência \(\{z_n\}\), tomamos uma outra sequência \(\{w_n\}\) em um outro subsetor, \(S''\), que também converge para 0 e repetimos o argumento anterior. Isso é feito usando novamente a Desigualdade do Valor Médio, mas considerando agora um subsetor suficientemente grande que contém os subsetores \(S'\) e \(S''\).
\end{proof}

\begin{obs} Combinando a proposição anterior e o lema anterior, obtemos estimativas semelhantes a \eqref{des:serie_gevrey}.
\end{obs}

Para \(s = 1\), usando o lema anterior podemos definir uma série formal \(\fs{f}(z) = \sum_{k=0}^\infty a_kz^k!/k!\) com \(a_k\) definidos como em \eqref{lem:termo_taylor}. Claramente essa série é Gevrey-1 e então, da observação \eqref{gevrey1holomorfa} temos estimativas de Cauchy e \(f\) pode ser definida holomorficamente em um disco centrado na origem. O caso que nos interessa é aquele em que \(s > 1\).

Como aqui trabalharemos com funções que não estão necessariamente definidas na origem, o nome \textit{série de Taylor} não é adequado. Vamos introduzir o conceito de desenvolvimento assintótico.

\begin{dfn} Seja \(f\) uma função holomorfa em \(S\). Dizemos que a série formal
\[
	\fs{f}(z) = \sum_{k=0}^\infty \frac{a_k}{k!} z^k \in \C[[z]]
\]
é uma \textbf{expansão assintótica de ordem \(\boldsymbol{s} \geq 1\)} da função \(f\) em \(S\) se, para todo \(S' \ssubset S\), existem constantes \(M,C > 0\) tais que
\begin{equation}
\label{des:expansao_assintotica_gevrey}	
	\babs{ f(z) - \sum_{k=0}^{n-1} \frac{a_k}{k!} z^k } \leq MC^{n} n!^{s-1} |z|^n, \quad n \geq 0,~ z \in S'.
\end{equation}

Por brevidade, muitas vezes usaremos a seguinte notação
\[
	f \gapprox{s} \fs{f} \quad \text{em} \quad S.
\]
\end{dfn}

A próxima proposição nos dá uma caracterização muito útil do conjunto \(\Hos{s}(S)\) que acabamos definir.

\begin{prp}
\label{prp:caracterizacao_de_Os} Seja \(f\) uma função holomorfa em \(S\). A função \(f\) pertence a \(\Hos{s}(S)\) se, e somente se, existe uma sequência de números complexos \(\{a_k\}\) tal que a série \(\fs{f}(z) = \sum_{k=0}^\infty a_k z^k/k!\) é uma expansão assintótica de ordem \(s\) de \(f\) em \(S\).
\end{prp}

\begin{proof} \((\implies)\)
Sejam \(f\) uma função satisfazendo \eqref{des:funcao_gevrey} e \(S' \ssubset S\).

Para cada \(k \in \N\), seja \(a_k\) definido como em \eqref{lem:termo_taylor}. Usando a fórmula de Taylor com \(z,z_0 \in S'\), temos:
\[
\begin{aligned}
\left|f(z) - \sum_{k=0}^{n-1} \frac{f^{(k)}(z_0)}{k!}(z - z_0)^k\right| 
	& \leq \left|\int_0^1 \frac{(1-t)^{n-1}}{(n-1)!} f^{(n)}(z_0 + t(z - z_0))(z - z_0)^{n} \dint t \right| \\
	& \leq \frac{|z - z_0|^n}{(n-1)!} \int_0^1 \left|f^{(n)}(z_0 + t(z - z_0))\right| \dint t \\
  & \leq \frac{MC^n n!^s}{(n-1)!}|z - z_0|^n. \\
\end{aligned}
\]


O resultado segue tomando o limite \(z_0 \to 0\) em \(S'\).

\((\impliedby)\)  Sejam \(S',S'' \ssubset S\) setores satisfazendo \(S' \ssubset S'' \ssubset S\) e \(\delta > 0\) tal que, para todo \(z \in S'\), \(\Disc(z, \delta|z|) \subset S''\). Aplicando a Fórmula de Cauchy, temos:

\[
\begin{aligned}
f^{(n)}(z) & = \frac{n!}{2\pi i}\int_{|z - \xi| = \delta|z|} \frac{f(\xi)}{(\xi - z)^{n+1}} \dint \xi \\
& = \frac{n!}{2\pi i}\int_{|z - \xi| = \delta|z|} \frac{f(\xi) - \sum_{k=0}^{n-1} a_k \xi^k/k!}{(\xi - z)^{n+1}} \dint \xi.
\end{aligned}
\]

Observemos que a última igualdade segue de:
\[
	\int_{|z - \xi| = \delta|z|} \frac{\sum_{k=0}^{n-1} a_k \xi^k/k!}{(\xi - z)^{n+1}} \dint \xi = 0.
\]

Sejam \(\sum_{k=0}^\infty a_k z^k/k!\) uma expansão assintótica de ordem s de \(f\) e \(M,C\) constantes positivas tais que \eqref{des:expansao_assintotica_gevrey} vale com \(S''\) no lugar de \(S'\). Temos:

\begin{equation}
\label{des:derivada_gevrey}
  \begin{aligned}
	  \babs{ f^{(n)}(z) }
		  & \leq \frac{n!}{(\delta|z|)^n} \sup_{|\xi - z| = \delta|z|} \babs{ f(\xi) - \sum_{k=0}^{n-1} \frac{a_k}{k!} \xi^k } \\
		  & \leq \frac{n!}{(\delta|z|)^n} \sup_{|\xi - z| = \delta|z|} MC^nn!^{s-1}|\xi|^n \\
		  & \leq MC^n(1 + 1/\delta)^nn!^s.
  \end{aligned}
\end{equation}
\end{proof}

A proposição anterior nos dá uma espécie de dicionário, relacionando funções holomorfas com séries que não são necessariamente convergentes.

Usando o lema \eqref{cap:des:limorigem}, podemos definir a seguinte aplicação.

\begin{dfn}
\label{cap-dev:aplicacao_de_taylor} Definimos a \textbf{aplicação de Taylor de classe \(\boldsymbol{s}\)}
	\[
	\label{def:tays}
		\tays{s}: \Hos{s}(S) \to \Cs{s}[[z]],
	\]
por
	\[
		\tays{s}(f) = \sum_{k=0}^\infty a_kz^k/k!
	\]
com \(a_k\) dados por \eqref{lem:termo_taylor}.
\end{dfn}

Para aplicações, é importante que saibamos se operações como a soma, o produto de duas funções e a derivação não muda o espaço em que as funções pertencem. Tais resultados são garantidos pela proposição a seguir.

\begin{prp}
\label{hosehumaalgebra} O conjunto \(\Hos{s}(S)\) munido da soma e do produto usuais é uma álgebra sobre \(\C\) fechada para a derivação.
\end{prp}

\begin{proof} Sejam \(f,g \in \Hos{s}(S)\) 
e \(S' \ssubset S\) um subsetor qualquer. Sejam também \(M,C > 0\) constantes tais que \eqref{des:expansao_assintotica_gevrey} vale para \(f\) e \(g\).

Segue da Desigualdade Triangular que \(f+g \in \Hos{s}(S)\). Vamos mostrar que \(fg \in \Hos{s}(S)\). Da regra de Leibniz, temos:
\[
\begin{aligned}
	|(fg)^{(n)}(z)| 
	& \leq \sum_{k=0}^n \binom{n}{k} MC^{n-k}(n-k)!^s MC^{k}k!^s \\
	& = M^2C^n\sum_{k=0}^n \binom{n}{k} [(n-k)!k!]^s \\
	& = M^2C^n n!^s \sum_{k=0}^n \binom{n}{k}^{1-s}. \\
\end{aligned}
\]

O resultado segue usando que \(\binom{n}{k}^{1-s} \leq 1\) e que \(n \leq 2^n\).

Resta-nos provar que \(f' \in \Hos{s}(S)\).
	\[
		|(f')^n(z)| = |f^{n+1}(z)| \leq MC^{n+1}(n+1)!^s \leq MC^{n+1}(2^s)^{n+1}n!^s, \quad n \geq 0,~ z \in S'.
	\]
 Logo \(f' \in \Hos{s}{S}.\) Como queríamos.

%

\end{proof}

Agora que estabelecemos que \(\Hos{s}(S)\) é uma álgebra fechada para as derivações, podemos enunciar a seguinte proposição que será importante para aplicações.

\begin{prp}
\label{homo} A aplicação de Taylor \[\tays{s}: \Hos{s}(S) \to \Cs{s}[[z]]\] é um homomorfismo de álgebras que comuta com as derivações.
\end{prp}

A demonstração é simples e segue direto da definição.

\section{O Lema de Watson}

Vamos estudar a aplicação \(\tays{s}\) que definimos anteriormente. Veremos, mais especificamente, sob quais situações essa aplicação é injetora e sob quais ela é sobrejetora.

Encontrando condições necessárias para a injetividade da aplicação de Taylor, temos também condições necessárias para que exista uma única soma da série. Por outro lado, ao determinarmos condições para a não injetividade da aplicação, temos condições suficientes para que a série não possua uma única soma.

Começaremos estudando a injetividade. Para isso, precisamos usar um resultado importante e útil no estudo de Análise Assintótica: o Lema de Watson. Com ele, conseguiremos encontrar condições necessárias para a aplicação \(\tays{s}\) ser injetora.

\begin{lem}[de Watson] Sejam \(\Omega = \{ z \in \C; \Re z > c\}\) e \(f\) uma função holomorfa em \(\Omega\). Se existem constantes positivas \(A\) e \(B\) tais que, para todo \(z\) em \(\Omega\), vale que \(|f(z)| \leq Ae^{-B|z|}\), então \(f\) é identicamente nula.
\end{lem}

\begin{proof} Como translações não mudam o decaimento exponencial da função \(f\), podemos supor que \(c < 0\). Seja \(I\) o eixo imaginário de \(\C\).

Para cada \(t \in \R\), definimos

\[
	g(t) \doteq \int_I f(z)e^{zt}\dint z = i\int_{-\infty}^\infty f(ix)e^{ixt}dx.
\]

Vamos provar que \(g\) é nula em \(\R\) e usando Análise de Fourier, vamos concluir que a função \(f\) é ser nula em \(I\).

Para cada \(\theta \in [-\pi/2, \pi/2]\), considere a curva
\[
	\gamma_\theta(x) \doteq xe^{i\theta}, \quad x \in [0,\infty).
\]

Podemos definir, para \(t \in \R\) com \(\Re(te^{i\theta}) = t\cos(\theta) < B\), a função \(h_\theta\) dada pela seguinte integral:

\[
	h_\theta (t) \doteq \int_{\gamma_\theta} f(z)e^{zt}\dint z = \int_0^\infty f(xe^{i\theta})e^{txe^{i\theta}}e^{i\theta} \dint x.
\]

É importante notarmos que, se \(w \in \gamma_\theta\), então \(|f(w)e^{wt}| \leq Ae^{-Bx}e^{x\Re(te^{i\theta})}\). Assim a condição \(\Re(te^{i\theta}) < B\) garante que a integral converge e a função \(h_\theta\) está bem definida.

Vamos provar que, para quaisquer \(\theta \in (0,\pi/2)\) e \(t < \frac{B}{\cos \theta}\), temos \(h_\theta(t) = h_{-\theta} (t)\).

Como \(f(z)e^{zt}\) é holomorfa em \(\Omega\), vale que
\[
0 = - \int_{\gamma_{\theta,r}} f(z)e^{zt}\dint z + \int_{\psi_{\theta,r}} f(z)e^{zt}\dint z + \int_{\gamma_{-\theta,r}} f(z)e^{zt}dz,
\]

onde \(\gamma_{\theta,r}\) é a curva \(\gamma_{\theta}\) restrita ao intervalo \([0,r]\) e \(\psi_{\theta,r}\) é o segmento de reta que começa em \(\gamma_{-\theta}(r)\) e termina em \(\gamma_{\theta}(r)\). Mais precisamente, para \(x \in [0,1]\),

\[
	\psi_{\theta,r}(x) = re^{-i\theta} + x[re^{i\theta} - re^{-i\theta}] = r[\cos\theta + i\sen\theta(2x -1)].
\]

Como
\[
	h_\theta(t) = \lim_{r \to \infty} \int_{\gamma_{\theta,r}} f(z)e^{zt}\dint z,
\]

nos resta mostrar que
\[
	\lim_{r \to \infty} \int_{\psi_{\theta,r}} f(z)e^{zt}\dint z = 0.
\]

Valem as seguintes relações:
\begin{itemize}
	\item \(\Re(\psi_{\theta,r}(x)) = r\cos(\theta);\)
	\item \(| \psi_{\theta,r}(x)  |^2 = r^2[\cos^2\theta + \sen^2\theta(2x - 1)^2] \leq r^2;\) e
	\item \(\psi'_{\theta,r}(x) = i2r\sen(\theta)\).
\end{itemize}

Portanto
\[
\begin{aligned}
\left|\int_{\psi_{\theta,r}} f(z)e^{zt}dz \right|
	 = & \left| \int_0^1 f(\psi_{\theta,r}(x))e^{t\psi_{\theta,r}(x)}\psi'_{\theta,r}(x) \dint x  \right| \\
\leq & \int_0^1 Ae^{-B|\psi_{\theta,r}(x)|} e^{t\Re(\psi_{\theta,r}(x))}|\psi'_{\theta,r}(x)| \dint x \\
\leq & \int_0^1 Ae^{-Br} e^{tr\cos\theta} 2r\dint x \\
   = & Ae^{r(t\cos\theta -B)}2r.
\end{aligned}
\]

Como \(t\cos\theta -B < 0 \), temos que \(\lim_{r \to \infty} Ae^{r(t\cos\theta -B)}2r = 0\) e segue que \(h_\theta(t) = h_{-\theta}(t)\).

Para concluirmos que \(g\) é nula, basta mostrarmos que \(g(t) = h_{\pi/2}(t) - h_{-\pi/2}(t)\). De fato, sejam \(t\) fixado e \(\theta\) suficientemente próximo de \(\pi/2\) tal que \(t\cos\theta < B/2\) . Escrevendo \(w = xe^{i\theta}\), vale que \(|w| = x\) e que

\[
 |f(w)e^{wt}| \leq Ae^{-B|w|}e^{t|w|\cos\theta} = Ae^{x(t\cos\theta -B)} \leq Ae^{-xB/2}.
\]

Do Teorema da Convergência Dominada de Lebesgue, obtemos:
\[
\begin{aligned}
	\lim_{\theta \to \pi/2} h_\theta(t) & = \lim_{\theta \to \pi/2} \int_0^\infty f(xe^{i\theta})e^{txe^{i\theta}}e^{i\theta} \dint x \\
	& = \int_0^\infty f(ix)e^{tix}i \dint x.
\end{aligned}
\]

Igualmente calculamos:
\[
\begin{aligned}
	\lim_{\theta \to -\pi/2} h_\theta(t) & = \int_0^\infty f(xe^{i\theta})e^{txe^{i\theta}}e^{i\theta} \dint x \\
	& = \int_0^\infty -f(-ix)e^{-tix}i \dint x \\
	& = -\int_{-\infty}^0 f(ix)e^{tix}i \dint x.
\end{aligned}
\]

Como \(h_\theta(t) - h_{-\theta}(t) = 0\), chegamos à seguinte conclusão:
\[
\begin{aligned}
	0 & = \lim_{\theta \to \pi/2} h_\theta(t) - h_{-\theta}(t) \\
	& = i\int_0^\infty f(ix)e^{tix}\dint x + i\int_{-\infty}^0 f(ix)e^{tix}\dint x \\
	& = g(t).
\end{aligned}
\]

Agora vamos verificar que a função \(x \longmapsto f(ix)\) está em \(\Sw(\R)\). De fato, dados \(n,m \in \N\), temos que \(\frac{\dint^n}{\dint x^n} f(ix) = f^{(n)}(ix)i^n\). Com \(\varphi(t) = re^{it}\), para \(t \in [0,2\pi]\) e \(r = |c|/2\), segue da Fórmula Integral de Cauchy:
\[
\begin{aligned}
|f^{(n)}(ix)| & = \left| \frac{n!}{2\pi} \int_{\varphi} \frac{f(z)}{(z - ix)^{n+1}} \dint z\right| \\
  & \leq \frac{n!}{2\pi r^n} \int_0^{2\pi} |f(ix + re^{ti})| \dint t \\
 & \leq \frac{n!}{2\pi r^n} \int_0^{2\pi} Ae^{-B|ix + re^{ti}|} \dint t \\
 &  \leq \frac{n!}{2\pi r^n} \int_0^{2\pi} Ae^{-B||x| - r  |} \dint t \\
 & = \frac{n!Ae^{-B||x| - r|}}{r^n}.
\end{aligned}
\]

Vemos facilmente que:
\[
	\sup_{x \in \R} \left|x^m\frac{\dint^n}{\dint x^n} f(ix) \right|
  \leq \sup_{x \in \R} |x|^m\frac{n!Ae^{-B||x| - r|}}{r^n} < \infty.
\]

Finalmente \(g = 0\) implica que \(\frac{1}{2\pi} \int_{-\infty}^\infty f(ix)e^{ixt}\dint x = 0\). Como a transformada de Fourier \(\F:\Sw \to \Sw\) é uma bijeção, temos que \(f \equiv 0\) em \(I\). Como \(f\) é holomorfa, segue que \(f \equiv 0\) em \(\Omega\).
\end{proof}

Na literatura, é possível encontrar outras versões do Lema de Watson. Uma delas, encontrada em \cite{giuseppe2013onthevanishing}, é uma versão mais fraca, mas é uma bela aplicação do Princípio do Máximo.

\begin{lem}
\label{outrolemawatson} Seja \(S \subset \C\) um semiplano e \(f\) uma função holomorfa em \(S \cap U\), onde \(U\) é uma vizinhança aberta da origem. Suponhamos \(f\) se estende continuamente até a borda de \(S\) perto da origem e que para qualquer \(\lambda \geq \lambda_0\) vale:
\[
	|f(z)|e^{\lambda/|z|} \to 0 \quad \text{quando} \quad z \to 0,~ z \in S. 
\]

Nessas condições, \(f \equiv 0\).
\end{lem}

Esse lema é uma consequência direta do Lema de Watson. De fato, como \(|f(z)|e^{\lambda_0/|z|} \to 0\) quando \(z \to 0\) em \(z \in S\), existe uma constante \(C > 0\) tal que \(|f(z)| \leq Ce^{-\lambda_0/|z|}\). Portando, se definirmos \(g(z) = f(1/z)\) para \(z^{-1} \in S \cap U\), então a função \(g\) está nas condições do Lema de Watson e segue que \(g \equiv 0\). Logo \(f \equiv 0\).

\begin{proof} Como rotações e translações não alteram o decrescimento exponencial, podemos assumir que \(S\) é o semiplano \(\{z \in \C: \Re z > 0\}\). Sem perda de generalidade, também podemos assumir que \(D = D_1(0) \subset U\).

Para \(\lambda \geq \lambda_0\), definimos \(g_\lambda(z) = f(z)e^{\lambda/z}\). Então \(g_\lambda\) é holomorfa e contínua até a borda de \(S\). De fato, a afirmação pode ser verificada diretamente para \(z \in \del S\) tais que \(\Im z \neq 0\). Além disso, por hipótese temos
\[
	\babs{ f(z)e^{\lambda/z} } = \babs{ f(z) }e^{\Re \lambda/z} \leq \babs{ f(z) }e^{\lambda/|z|} \to 0, \quad \text{ quando } z \to 0,
\]
então basta definir \(g_\lambda(0) = 0\).

Definimos \(M \doteq \max_{z \in \overline{S\cap D}} |f(z)|\). Como \(g_\lambda\) é holomorfa e contínua até a fronteira de \(S \cap D\), podemos aplicar o princípio do máximo e obter
\begin{equation}
\label{petronilho}
	\max_{z \in \overline{S\cap D}} g_\lambda(z) = \max_{z \in \del(S\cap D)} |g_\lambda(z)| \leq Me^\lambda,
\end{equation}

pois \(\babs{ e^{\lambda/z} } = 1\) para \(z \in \del S\) e \(\babs{e^{\lambda/z}} \leq e^{\lambda/|z|} = e^\lambda\) para \(z \in \del D\).

Se tomarmos, \(r \in \R\) com \(0 < r < 1/2\), temos
\[
	|g_\lambda(r)| = |f(r)|e^{\Re \lambda/r} \geq |f(r)|e^{2\lambda},
\]
que combinado com \eqref{petronilho}, resulta em  \(|f(r)|e^\lambda \leq M\). Como \(\lambda\) pode ser tomado arbitrariamente grande, temos que \(f(r) = 0\) e assim \(f \equiv 0\).
\end{proof}

\section{Estudo da Aplicação de Taylor}

Usando o Lema de Watson conseguiremos provar a injetividade de \(\tays{s}\) em setores suficientemente grandes. Mas antes de provar a injetividade, vamos primeiro entender melhor o núcleo do operador \(\tays{s}\).

\begin{prp} Sejam \(s > 1\) e \(f \in \Hos{s}(S)\). Temos \(f \in \ker \tays{s}\) se, e só se, para todo setor \(S' \ssubset S\), existem constantes \(B, b > 0\) tais que
	\begin{equation}
		\label{des:ker_tays}
		|f(z)| \leq Be^{ -b|z|^{-\frac{1}{s - 1}} }, \quad z \in S'.
	\end{equation}
\end{prp}

\begin{proof} (\(\implies\)) Se \(f \in \ker \tays{s}\), os coeficientes da expansão assintótica de \(f\) em \(S\) são todos nulos. Da proposição \eqref{prp:caracterizacao_de_Os} segue que, para todo \(S' \ssubset S\), existem constantes positivas \(M\) e \(C\) tais que
\[
	\babs{ f(z) } \leq MC^{n} n!^{s-1} |z|^n, \quad n \geq 0,~ z \in S'.
\]
logo
\[
\babs{ f(z) } \leq M \min_{n \geq 0} \{C^{n} n^{n(s-1)} |z|^n\}, \quad z \in S'.
\]
Usando Cálculo Diferencial, vemos facilmente que o mínimo da função \(\lambda(t) = (C|z|)^t t^{t(s-1)}\), para \(t > 0\), é atingido em \(t_0 = (C|z|)^{-\frac{1}{s - 1}}/e\). Consequentemente, escrevendo \(t_0^* = [t_0]\) ou \(t_0^* = [t_0] + 1\), temos que
\[
	\min_{n \geq 0} \{C^{n} n^{n(s-1)} |z|^n\} = \lambda(t_0^*) = e^{t_0^* \log\bpar{ C|z|(t_0^*)^{s-1} }}.
\]

Como \(t_0^* \leq t_0 + 1\) e 
, na região 
\[
	S'' = \left\{z \in S': |z| \leq \frac{1}{e^{s-1}C} \right\}
\]
temos a estimativa \(t_0^* \leq 2(C|z|)^{-\frac{1}{s - 1}}/e\).

Portanto, se \(z \in S''\),
\[
	\log\bpar{ C|z|(t_0^*)^{s-1} } \leq \log \bpar{ 2^{s-1}/e^{s-1} } \doteq -a < 0.
\]

Assim
\( \babs{ f(z) } \leq Me^{-at^*_0} \). Como \(t^*_0 \geq t_0 - 1\), obtemos \( \babs{ f(z) } \leq Me^ae^{-at_0} = M'e^{-at_0}\). Finalmente
\[
	\babs{ f(z) } = M'e^{ -a(C|z|)^{-\frac{1}{s - 1}}/e }.
\]

 (\(\impliedby\)) Sejam \(S' \ssubset S'' \ssubset S\) e \(B,b > 0\) constantes tais que vale \eqref{des:ker_tays}. Se \(\delta > 0\) é tal que \(D_{\delta|z|}(z) \subset S''\), para todo \(z \in S'\), então podemos aplicar a Fórmula Integral de Cauchy:
\[
	\begin{aligned}
		|f^{(k)}(z)|
			& \leq \frac{k!}{(\delta|z|)^k} \sup_{|\xi - z| = \delta|z|} \babs { f(\xi) } \\
			& \leq \frac{k!}{(\delta|z|)^k} \sup_{|\xi - z| = \delta|z|} Be^{ -b|\xi|^{-\frac{1}{s-1}} } \\ 
			& \leq \frac{k!}{(\delta|z|)^k} Be^{ -b[|z|(1-\delta)]^{-\frac{1}{s-1}} }  .
	\end{aligned}
\]

Da estimativa acima, vemos que \(f^{(k)}(z) \to 0\) quando \(z \to 0\) em \(S'\). Portanto, do lema \eqref{cap:des:limorigem}, segue que \(f \in \ker \tays{s}\).
\end{proof}

\begin{prp} Se \(1 < s < 3\) e se o setor \(S\) tem abertura maior que \((s - 1)\pi\), então a aplicação \(\tays{s}\) é injetora.
\end{prp}

\begin{proof} Sejam \(f \in \ker \tays{s}\) e \(S' \ssubset S\) com abertura maior que \((s - 1)\pi\). Podemos escrever \(S' = \{ z \in \C: \alpha' < \arg z < \beta', 0 < |z| < R' \}\) com \(\beta'- \alpha' > (s - 1)\pi\). Definindo
\[
	\Omega \doteq \{w \in C: |w| > 1/R'^{\frac{1}{s-1}}, \beta'/(1 - s) < \arg w < \alpha'/(1 - s)\}
\]
temos que \(\alpha'/(1 - s) - \beta'/(1 - s) > (\beta' - \alpha')/(s - 1) > \pi\). Logo \(\Omega\) é simplesmente conexo e contém um semiplano. Isso nos permite definir holomorficamente a aplicação \(w \in \Omega \mapsto w^{1 - s} \in S'\). Da proposição anterior vale a estimativa
\[ \babs{ f(w) } \leq Be^{-b|w|}, \quad w \in S'. \]

Logo podemos aplicar o Lema de Watson à aplicação \(w \in S \mapsto f(w^{1-s})\) e concluir que \(f\) se anula identicamente.
\end{proof}

\subsection{Transformadas de Borel e Laplace}

Vamos definir duas transformadas: a transformada de Borel e a transformada de Laplace. Elas serão usadas não apenas para determinar se uma série formal é somável, mas também para encontrar uma fórmula da ``soma'' da série.

Começaremos demonstrando um lema técnico:

\begin{lem}
\label{lem:formula_nao_de_hankel} Seja \(w \in \C\) com \(\Re w > 0\). Para todo natural \(n\) e \(s > 0\), vale a seguinte igualdade:
\[
	\frac{1}{w^{ns}} = \frac{1}{\Gamma(ns)} \int_0^\infty t^{ns -1} e^{-tw} \dint t.
\]
\end{lem}

\begin{proof}
Para \(w \in \C\), denotaremos por \(w \R_+\) o conjunto
\[\{z \in \C: z = wx,~x \in \R_+\}.\]

Inicialmente notemos que
\[
\begin{aligned}
	\int_0^\infty t^{ns -1} e^{-tw} \dint t & = w^{-ns}\int_0^\infty (wt)^{ns -1} e^{-tw}w \dint t \\
	& = w^{-ns} \int_{w\R_+} z^{ns -1} e^{-z} \dint z.
\end{aligned}
\]

Basta mostrarmos que
\[
	\int_{w \R_+} z^{ns -1} e^{-z}  \dint z = \int_{\R_+} z^{ns -1} e^{-z}  \dint z.
\]

Sejam \(\alpha_R(t) = wt\) e \(\beta_R(t) = t\) definidas para \(t \in [1/R,R]\), com \(R > 1\) e \(\gamma_R(t) = Re^{it}\), para \(t \in [0,\arg(w)]\). Do Teorema de Cauchy, vale
\[
	\int_{\beta_R + \gamma_R - \alpha_R - \gamma_{1/R}} z^{ns -1} e^{-z}  \dint z = 0.
\]

Do Teorema da Convergência Dominada, obtemos
\[
	\int_{\gamma_R} z^{ns -1} e^{-z}  \dint z = \int_0^{\arg{w}} (Re^{it})^{ns -1} e^{-R e^{it}} iRe^{it}  \dint t \xrightarrow[R \to \infty ]{} 0
\]
e
\[
	\int_{\gamma_{1/R}} z^{ns -1} e^{-z}  \dint z = \int_0^{\arg{w}} (e^{it}/R)^{ns -1} e^{-e^{it}/R} ie^{it}/R  \dint t \xrightarrow[R \to \infty ]{} 0.
\]

Logo \[
	\lim_{R \to \infty} \int_{\beta_R} z^{ns -1} e^{-z} \dint z = \lim_{R \to \infty} \int_{\alpha_R} z^{ns -1} e^{-z}  \dint z
\]
e concluímos a demonstração.
\end{proof}


 

No exemplo a seguir vamos apresentar uma transformada (formal) de Borel e ver como ela, em combinação com a transformada de Laplace, pode ser usada para somar a série de Euler.

\begin{exe}
\label{exe:soma_de_euler} A série de Euler \(\fs{f} (z)= \sum_{k = 1}^\infty (-1)^{k-1}(k-1)! z^k\) é de classe de Gevrey-2. Portanto, se dividirmos cada coeficiente desta série por \((k-1)!\), obtemos uma série formal \(\hat{g}\), que é Gevrey-1, isto é, uma função holomorfa 
\begin{equation}
\label{naoboreltransform}
	\hat{g}(z) = \sum_{k = 1}^\infty (-1)^{k-1} z^k = \frac{z}{1 + z},
\end{equation}
que está definida para todo \(z \not= -1\). À série formal \(\hat{g}\), damos o nome de transformada formal de Borel de \(\fs{f}\).

Para \(\Re z > 0\), definimos a função
\begin{equation}
\label{soma}
  f(z) = \ddint{0}{\infty}{ \frac{ \hat{g}(t) }{t} e^{-t/z} }{t} = \ddint{0}{\infty} { \frac{ 1 }{1 + t} e^{-t/z} }{t}.
\end{equation}

Esta função \(f\) admite \(\fs f\) como expansão assintótica de ordem 2 em \(S(0,\pi,r)\), para qualquer \(r > 0\).

De fato, derivando sob o sinal de integração, vemos que \(f\) é holomorfa em \(S\). Além disso, como
\[
  \frac{1}{1+t} = \sum_{k=1}^n (-1)^{k-1}t^{k-1} + (-1)^n\frac{t^n}{1 + t},
\]
vale que
\[
  \ddint{0}{\infty} { \frac{ 1 }{1 + t} e^{-t/z} }{t} = \sum_{k=1}^{n-1} (-1)^{k-1}\ddint{0}{\infty}{ t^{k-1} e^{-t/z}} {t} + (-1)^n \ddint{0}{\infty}{ \frac{t^n}{1+t} e^{-t/z} } {t}.
\]

Da lema \eqref{lem:formula_nao_de_hankel}, temos
\[
  \ddint{0}{\infty}{ t^{k-1} e^{-t/z}} {t} = (k-1)!z^{k}.
\]

Assim, para cada natural \(n \geq 2\), vale
\[
  \ddint{0}{\infty} { \frac{ 1 }{1 + t} e^{-t/z} }{t} - \sum_{k=1}^{n-1} (-1)(k-1)!z^{k} = (-1)^n \ddint{0}{\infty}{ \frac{t^n}{1+t} e^{-t/z} } {t}.
\]

Dado um subsetor \(S' \ssubset S\), tomemos \(\delta > 0\) tal que \(S'' \subset \{z \in \C: |\arg z| < \pi/2 - \delta\}\). Para \(z \in S''\), vale que \(\babs{ e^{t/z} } \leq e^{t\sen\delta/|z|}\) e segue a estimativa
\[
	\begin{aligned}
  \babs{ \ddint{0}{\infty}{ \frac{t^n}{1+t} e^{-t/z} } {t} } & \leq \ddint{0}{\infty}{ \frac{t^n}{1+t} e^{-t\sen\delta/|z|} } {t} \\ & \leq \ddint{0}{\infty}{ t^n e^{-t\sen\delta/|z|} } {t} \\ & =  \frac{n! |z|^{n+1}}{(\sen\delta)^{n+1} },
  \end{aligned}
\] com a última igualdade obtida usando novamente o lema \eqref{lem:formula_nao_de_hankel}. Como queríamos. 

\end{exe}

Como curiosidade, podemos encontrar no artigo de V.S. Varadarajan \cite{varadarajan2007euler}, onde é exposto um pouco sobre o trabalho de Euler em séries infinitas, algumas ideias de como Euler chegou à atribuição \[
	\sum_{k=0}^\infty (-1)^k k! = 0,59637255\ldots.
\]

Uma pergunta natural a se fazer é: qual é o valor da função definida em \eqref{soma} em \(1\)? Usando integração numérica, obtemos
\[
	\ddint{0}{\infty} { \frac{ 1 }{1 + t} e^{-t} }{t} = 0,59637255\ldots,
\]
que é um resultado impressionante dado que naquela época não havia nem uma definição precisa do conceito de convergência.

Com base no exemplo acima, queremos usar as transformada de Laplace e de Borel para estudar séries divergentes. Porém, serão necessárias algumas adaptações. 

\begin{obs}
Na literatura, a série formal obtida em \eqref{naoboreltransform} é conhecida como transformada de Borel formal, porém não usaremos essa definição. Precisamos fazer modificações na transformada de Borel e na transformada de Laplace de acordo com cada ordem de Gevrey.

Além disso, a transformada de Laplace usual possui a característica de diminuir a ordem de polinômios, o que pode tornar as demonstrações um pouco mais complicadas. Portanto, o que faremos a seguir é apresentar as transformadas de Borel e transformada de Laplace adaptadas para as nossas necessidades.
\end{obs}

A seguir, vamos definir uma classe de funções que serão importantes por dois motivos: o primeiro é que a transformada de Laplace está bem definida nessa classe de funções; o segundo é que, com essas funções, conseguiremos dar uma condição necessária e suficiente para uma série formal ser somável.

\begin{dfn} Sejam \(S\) um setor de raio infinito e \(f\) uma função holomorfa em \(S \bs D_\rho\), para algum \(\rho \geq 0\). Dizemos que \(f\) tem crescimento exponencial de ordem \(m > 0\) em \(S\) se, para todo subsetor \(S' \ssubset S\), existem constantes positivas \(r,b,B\), com \(r > \rho\), tais que
\[
	|f(z)| \leq Be^{ b|z|^m }, \quad z \in S', \quad |z| > r.
\]
\end{dfn}



Agora estamos prontos para definir uma família de transformadas de Laplace. Será importante definir uma transformada que será adequada para trabalhar com cada classe de Gevrey.

\begin{lem}
\label{lem:transformada_de_laplace} Sejam \(S = S(\theta, \alpha)\) e \(f\) holomorfa e com crescimento exponencial de ordem \(m\) em \(S\). Para \(\eta\) satisfazendo \(|\theta - \eta| < \alpha\), a integral
\begin{equation}
\label{eq:temporaria1}
	\int_0^{\infty e^{i\eta}} f(u)e^{ -(u/z)^m } mu^{m-1} \dint u,
\end{equation}
com a integração sendo feita na semirreta \(\arg u = \eta\), converge absolutamente no aberto dado pela desigualdade
\[
	0 < b|z|^m < \cos[m(\eta - \arg z)],
\]
onde \(b > 0\) é uma constante que depende de \(f\) e \(\eta\). Nessa região, a função
\begin{equation}
\label{eq:transformada_de_laplace}
	g(z) \doteq \frac{m}{z^m}\int_0^{\infty e^{i\eta}} f(u) e^{ -(u/z)^m } u^{m-1}\dint u
\end{equation}
é holomorfa.
\end{lem}

\begin{proof} Primeiro vamos provar que a integral \eqref{eq:temporaria1} converge absolutamente.

Como \(|\theta - \eta| < \alpha/2\) e \(f\) tem crescimento exponencial de ordem \(m\), existem constantes positivas \(r,b\) e \(B\) tais que
\[
	|f(u)| \leq Be^{\frac{b}{2}|u|^m}, \quad \arg u = \eta, \quad |u| \geq r.
\]

Portanto, ainda considerando \(\arg u = \eta\) e \(|u| \geq r\),
\[
	 \babs{ f(u)e^{ -(u/z)^m } mu^{m-1} } \leq Be^{ \frac{b}{2}|u|^m -\Re (u/z)^m } m|u|^{m-1}
\]
e, como
\[
\Re (u/z)^m = \frac{|u|^m}{|z|^m} \cos[m(\eta - \arg z)],
\]
para \(z\) satisfazendo
\[
	b|z|^m < \cos[m(\eta - \arg z)], 
\]
temos que
\[
	\int_0^\infty Be^{ \frac{b}{2}t^m -t^m \cos[m(\eta - \arg z)] } mt^{m-1} \dint t \leq \int_0^\infty Be^{ -\frac{b}{2}t^m} mt^{m-1} \dint t < \infty.
\]

Logo a integral \eqref{eq:temporaria1} converge absolutamente.

A afirmação de que a função definida em \eqref{eq:transformada_de_laplace} é holomorfa segue usando as equações de Cauchy-Riemann e derivação sob o sinal de integração.
\end{proof}

\begin{obs}
Uma pequena variação no parâmetro \(\eta\) não muda o valor da integral, mas apenas a região de convergência. Essa variação de \(\eta\) dá uma continuação analítica de \(g\), assim omitiremos o parâmetro \(\eta\).
\end{obs}

\begin{dfn} Nas mesmas condições do Lema \ref{lem:transformada_de_laplace}, usamos a notação \(g = \La_m f\) e dizemos que \(g\) é a \textbf{transformada de Laplace de índice m} de \(f\).
\end{dfn}

O próximo exemplo mostra o que acontece quando usamos transformada de Laplace em polinômios. Como a transformada de Laplace é linear, precisamos apenas calcular \(\La_{m}\) em monômios.

\begin{exe}
Se \(f(z) = z^\lambda\), com \(\Re \lambda > 0\), então
\[
	\bpar{ \La_m f }(z) = \Gamma(1 + \lambda/m) z^\lambda.
\]
De fato, para qualquer \(\eta \in \R\), temos
\[
	\begin{aligned}
		g(z) & = z^{-m}\int_0^\infty (te^{i\eta})^\lambda e^{ -(te^{i\eta}/z)^m } m (te^{i\eta})^{m-1} e^{i\eta} \dint t \\
			& = z^{\lambda}\int_0^\infty y^{\lambda/m}e^{-y} \dint y \\
			& = z^{\lambda}\Gamma(1 + \lambda/m).
	\end{aligned}
\]
Na segunda igualdade usamos a mudança de variável \(y = (te^{i\eta})^mz^{-m}\).
\end{exe}

O exemplo anterior mostrou o efeito da transformada de Laplace em polinômios. É natural questionar se é possível fazer o mesmo com séries, isto é, se aplicar a transformada de Laplace em uma função definida por uma série é o mesmo que aplicar a transformada termo a termo.

Como nem sempre podemos trocar o limite da soma com a transformada de Laplace, definiremos uma transformada de Laplace formal e depois estudaremos como a transformada de Laplace se relaciona com a transformada de Laplace formal.

\begin{dfn} Dada uma série formal \(\fs{f}(z) = \sum_{n=0}^\infty a_n z^n\), definimos a \textbf{transformada de Laplace formal de índice \(\boldsymbol{m}\)} pela fórmula
\[
	\fLa_m(\fs{f})(z) \doteq \sum_{n=0}^\infty a_n\Gamma(1 + n/m) z^n.
\]
\end{dfn}

O próximo teorema é a primeira informação que temos sobre como a transformada de Laplace se relaciona com a transformada de Laplace formal.

\begin{teo}
\label{laplaceassint} Sejam \(f\) holomorfa e de crescimento exponencial de ordem \(m\) em um setor \(S = S(\theta, \alpha)\) e \(g = \La_{m}f\). Para \(s \geq 1\), suponhamos que
\[
	f \gapprox{s} \fs{f} \quad \text{ em } \quad S,
\]
onde \(\fs{f}(z) = \sum_{k=0}^\infty a_k z^k/k! \in \C_{(s)}[[z]]\).

Se definirmos \(\tilde{s} = m^{-1} + s\) e tomarmos \(\fs{g}(z) = (\fLa_{m}\fs{f})(z) = \sum_{k=0}^\infty b_k z^k/k!\), então, para todo \(\epsilon > 0\), existe \(\rho = \rho(\epsilon) > 0\) tal que \(g\) é holomorfa em
\(S_\epsilon = S(\theta, \alpha + \pi(s - 1) - \epsilon,\rho)\)
e
\[
		g \gapprox{\tilde{s}} \fs{g} \quad  \text{ em } \quad S_\epsilon.
\]


\end{teo}

\begin{proof} Para um \(\delta > 0\), com \(\delta < \alpha\), consideremos os setores \(S_\delta = S(\theta, \alpha-\delta)\) e \(S_{\delta,1} = S(\theta, \alpha - \delta, 1)\).

Como \(f\) tem crescimento exponencial de ordem de ordem \(m\), existem constantes positivas \(B\) e \(b\) tais que
\[
	|f(z)| \leq Be^{b|z|^m}, \quad z \in S_\delta.
\]

Como \(f \gapprox{s} \fs{f}\) em \(S\), sabemos que \(\fs{f}\) é Gevrey-\(s\), então existem constantes positivas \(M\) e \(C\) tais que
\[
	\babs{ f(z) - \sum_{k=0}^{n-1} \frac{a_k}{k!}z^k } \leq MC^n\Gamma(1 + (s - 1)n)|z|^n, \quad z \in S_{\delta,1},~n = 0, 1, 2, \ldots,
\]
e
\[
	\babs{ \frac{a_k}{k!} } \leq MC^k\Gamma(1 + (s - 1)k),\quad k = 0, 1, 2, \ldots.
\]

Como \(f\) tem crescimento exponencial, eventualmente aumentando as constantes \(M\) e \(C\), obtemos
\[
	\babs{ f(z) - \sum_{k=0}^{n-1} \frac{a_k}{k!}z^k } \leq MC^n\Gamma(1 + (s - 1)n)|z|^n, \quad z \in S_{\delta},~n = 0, 1, 2, \ldots,
\]

Temos que, para \(z\) satisfazendo \(b|z|^m < \cos[m(\theta - \arg z)]\), vale
\[
	\La_m \bbrac{ f(u) - \sum_{k=0}^{n-1} \frac{a_k}{k!}u^k }(z) = g(z) - \sum_{k=0}^{n-1} \frac{b_k}{k!}z^k, \quad n = 1, 2, \ldots ,
\]
portanto, dado \(\epsilon > 0\), existe um \(\rho = \rho(\epsilon) > 0\) tal que, se \(z \in S_\epsilon  \) então \(b|z|^m < \cos[m(\theta - \arg z)]\).

Logo, para \(z \in S_\epsilon\)
\[
	\begin{aligned}
		\babs{ g(z) - \sum_{k=0}^{n-1} \frac{b_k}{k!}z^k }
			& \leq \frac{m}{|z|^m}\int_0^\infty \babs{f(te^{i\theta}) - \sum_{k=0}^{n-1} \frac{a_k}{k!}(te^{i\theta})^k} e^{-\Re(te^{i\theta}/z)^m}t^{m-1}\dint t \\
			& \leq MC^n\Gamma(1 + (s - 1)n) \frac{m}{|z|^m} \int_0^\infty t^n e^{-t^m\cos[m(\theta - \arg z)]/|z|^m}t^{m-1}\dint t \\
			& = MC^n\bpar { \frac{|z|}{cos(\theta - \arg z)} }^n \Gamma(1 + (s - 1)n) \Gamma(1 + n/m).
	\end{aligned}
\]
A última igualdade é obtida usando a fórmula da transformada de Laplace calculada em monômios.

Agora basta escolher \(\delta < \epsilon\) e usar a Fórmula de Stirling que a demonstração está concluída. \end{proof}

A próxima proposição nos dá informações sobre o núcleo da transformada de Laplace.

\begin{prp}
\label{laplacekernel} Seja \(f\) uma função contínua em \(\{ u =r e^{i\theta}: 0 \leq r < \infty\}\). Se existem constantes positivas \(B,~b\) e \(m\) tais que
\[
	|f(u)| \leq Be^{b|u|^m}
\]
e, se 
\begin{equation}
\label{eqnuclap}
	\int_0^{\infty e^{i\theta}} f(u) e^{ -(u/z)^m } u^{m-1}\dint u = 0,
\end{equation}
para todo \(z\) satisfazendo
\[
	0 < b|z|^m < \cos(m[\theta - \arg z]),
\]
então \(f = 0\).

\end{prp}

\begin{proof} Como rotações não mudam o crescimento exponencial, podemos assumir, sem perda de generalidade, que \(\theta = 0\).

Sabemos que \eqref{eqnuclap} está definida e é holomorfa em
\[
	0 < b|z|^m < \cos(m[\theta - \arg z]).
\]



Fazendo a mudança de variável \(v = u^m\), a expressão acima pode ser escrita da seguinte forma:
\begin{equation}
\label{ortogonalapolinomios}
	\int_0^{\infty} f(v^{1/m}) e^{-vz^{-m}} \dint v = 0. 
\end{equation}

Escrevendo \(w = z^{-m}\), temos que a expressão acima está definida se
\[
	0 < b|w|^{-1} < \cos(-m\arg w^{-1/m}).
\]

Denotando \(w = x + iy\) para \(x > 0\), como \(-m\arg w^{-1/m} = \arg w\) temos que \(\arg w = \arccos(x/(x^2 + y^2)^{1/2})\) e a expressão acima está definida para
\[
	0 < \frac{b}{\sqrt{x^2 + y^2}} < \frac{x}{\sqrt{x^2 + y^2}},
\]
isto é,  para \(x > b\).

Fixando \(x = b + 1\), a expressão \eqref{ortogonalapolinomios} fica
\[
	0 = \int_0^{\infty} f(v^{1/m}) e^{-v(b+1) - ivy} \dint v = \F\{H(v)f(v^{1/m}) e^{-v(b+1)} \}(y),
\]
onde \(H\) é a função de Heaviside.

Como \(H(v)f(v^{1/m}) e^{-v(b+1)}\) está em \(L^1(\R)\) e a transformada de Fourier dessa função é nula, podemos usar a transformada de Fourier Inversa, e concluir que a função \(H(v)f(v^{1/m}) e^{-vb}\) se anula identicamente. Disso segue que \(f\) se anula em \(\R_+\).
\end{proof}


Agora vamos definir a transformada de Borel e a transformada de Borel formal. Vamos ver que essa transformada é ideal para trabalhar com funções definidas em setores de raio finito e que ela define uma função de crescimento exponencial definida em um setor de raio infinito, que é exatamente o tipo de função que precisamos para trabalhar com a transformada de Laplace. 

Depois veremos que, eventualmente diminuindo o domínio de definição das funções, a transformada de Laplace e a transformada de Borel são uma inversa da outra.

 Suponha que \(m > 1/2\) e defina \(\gamma_{m}(\eta) = \gamma_{m}(\eta,\epsilon,R)\) a curva obtida percorrendo no sentido anti-horário a fronteira do conjunto \(S(\eta,(\epsilon + \pi)/m, R)\), onde \(R,~\epsilon\) são constantes positivas com \(\epsilon\) suficientemente pequeno para que \((\epsilon + \pi)/m < 2\pi\).

\begin{dfn}
Sejam \(m > 1/2\) e \(S = S(\theta, \alpha, \rho)\) um setor de abertura \(\alpha\) maior que \(\pi/m\). Escolhemos \(\eta, \epsilon\) e \(R\) tais que \(\gamma_m(\eta)\) esteja contida em \(S\). Se \(f\) é holomorfa em \(S\) e limitada na origem, podemos definir a \textbf{transformada de Borel de índice \(\boldsymbol{m}\)} de \(f\) por
\begin{equation}
	\bpar{ \Bo{m} f }(w) \doteq \frac{m}{2 \pi i} \int_{\gamma_m(\eta)} f(z) e^{(w/z)^m} z^{-1}\dint z, \quad w \in S(\eta, \epsilon/m).
\end{equation}

Observemos que do Teorema de Cauchy segue que a definição da transformada de Borel não depende das escolhas de \(\epsilon\) e \(R\), portanto esses valores serão omitidos.


\end{dfn}

Nos próximos resultados, em especial no Teorema \ref{borelassint}, ficará claro que a transformada de Borel está bem definida.

Análogo ao que fizemos com a transformada de Laplace, podemos pensar em qual seria o efeito de aplicar a transformada de Borel termo a termo a uma série. 

Para a próxima definição precisamos de uma outra representação integral da função Gamma, conhecida como fórmula de Hankel,

\begin{equation}
	\frac{1}{\Gamma(z)} = \frac{i}{2\pi} \int_C (-t)^{z-1}e^{-t} \dint t,
\end{equation}
onde \(C\) é o caminho que começa no ``infinito'' da reta real, contorna a origem no sentido antihorário e depois volta ao ponto de origem.

\begin{dfn} Se \(f(z) = z^\lambda\), a fórmula de Hankel nos diz que \[
	(\Bo{m}f)(u) = \frac{u^\lambda}{\Gamma(1 + \lambda/m),}
\] portanto, aplicando a transformada de Borel termo a termo a uma série formal
\[
	\fs{f}(z) = \sum_{k=0}^\infty a_k z^k,
\]
obtemos
\begin{equation}
	(\fBo{m}{\fs{f}})(z) \doteq \sum_{n=0}^\infty \frac{a_k}{\Gamma(1 + k/m)}z^k,
\end{equation}
a \textbf{transformada de Borel formal (de índice \(\boldsymbol{m}\))} de \(\fs{f}\).
\end{dfn}

Claramente a transformada de Borel formal é a inversa da transformada de Laplace formal. Essa é a primeira sugestão de que a transformada de Borel e a de Laplace são uma inversa da outra.

Além disso, a série obtida da transformada de Borel é uma série convergente, fato que provamos no lema a seguir.

\begin{lem}
\label{lem:raio_conv_borel}  Se \(\fs{f} \in \Cs{s}[[z]]\) e m = \((s - 1)^{-1}\), então \(\fBo{m} \fs f \) possui raio de convergência positivo.
\end{lem}
 
\begin{proof} Suponhamos que \(\fs{f}(z) = \sum_{k=0}^\infty a_k z^k/k!\). Como \(\fs{f}\) é de classe \(s\), existem constantes \(C,M > 0\) tais que, para todo \(k\), vale a estimativa
\[
	\begin{aligned}
	  \frac{\babs{ a_k }}{k!} \frac{ |t|^k }{\Gamma(1 + k/m))} & \leq MC^k \Gamma(1 + (s-1)k) \frac{ |t|^k }{\Gamma(1 + k/m)} \\
	  & = MC^k|t|^k.
	\end{aligned}
\]

Logo
\[
  \begin{aligned}
    \babs{ \B{m} \fs f (t) }
      & \leq \sum_{k=0}^\infty \frac{ \babs{a_k} }{k!} \frac{ |t|^k }{\Gamma(1 + k/m)} \\
      & \leq M \sum_{k=0}^\infty (C|t|)^k,
  \end{aligned}
\]
que converge para \(|t| < 1/C\).
\end{proof}

Com as transformadas que construímos até agora, já podemos obter mais um resultado importante para a Análise Assintótica de Gevrey. Conseguimos, sob certas condições, dar uma expressão para uma ``soma'' de uma série.

\begin{teo}
\label{cap:des:injetividade} Sejam \(s \geq 1\), \(\fs{f} \in \Cs{s}[[z]]\) e \(S = S(\theta,\alpha)\) um setor de abertura menor ou igual a \(\min\{2\pi, (s-1)\pi\}\). Nessas condições, existe uma função \(f \in \Hos{s}(S)\) tal que
\(
	f \gapprox{s} \fs{f} \text{ em } S .
\)
\end{teo}

\begin{proof} Suponha que \(\fs{f}\) possa ser escrita como
\[
	\fs{f}(z) = \sum_{n=0}^\infty \frac{a_n}{n!} z^n.
\]

Queremos definir uma função holomorfa \(f \in \Ho_{(s)}(S)\) que tem \(\fs{f}\) como expansão assintótica.

Por hipótese, temos que \(\fs{f}\) é Gevrey-\(s\). Se \(s = 1\), já sabemos que a série converge e define uma função holomorfa em uma vizinhança da origem. Vamos considerar o caso \(s > 1\). Nesse caso, escrevendo \(m = (s-1)^{-1}\) e usando o lema anterior, sabemos que a função
\[
	g(z) = (\fBo{m} {\fs f})(z) = \sum_{n=0}^\infty \frac{b_n}{n!} z^n,\]
é holomorfa, pois é definida por uma série que converge absolutamente em um disco de raio positivo.

Para \(\epsilon > 0\) e \(\theta\), o ângulo da reta que bissecta o setor \(S\), se \(z\) é tal que \(\cos[m(\theta - \arg z)] \geq \epsilon\) e tomando \(\rho > 0\) suficientemente pequeno, podemos definir
\begin{equation}
\label{somadeborel}
	f^\rho(z) = \frac{m}{z^m} \int_0^{\rho e^{i\theta}} g(u)e^{-(u/z)^m } u^{m-1}\dint u.
\end{equation}

A convergência da série que define \(g\) é absoluta no disco de raio \(\rho\), portanto vamos usar a seguinte expressão para \(f^\rho\):

\[
	f^\rho(z) =  \sum_{k=0}^\infty \frac{b_k}{k!} \frac{m}{z^m} \int_0^{\rho e^{i\theta}} u^k e^{ -(u/z)^m } u^{m-1}\dint u.
\]

Vamos provar que \(f^\rho \gapprox{s} \fs{f}\) em \(S\).

Sabemos que vale a seguinte igualdade:
\[
\begin{aligned}
	\sum_{k=0}^{n-1}\frac{a_k}{k!}z^k = \frac{m}{z^m} \int_0^{\infty e^{i\theta}} \sum_{k=0}^{n-1} \frac{b_k}{k!} u^k e^{ -(u/z)^m } u^{m-1}\dint u. 
\end{aligned}
\]

Portanto, podemos escrever
\[
\begin{aligned}
 f^\rho(z) - \sum_{k=0}^{n-1}\frac{a_k}{k!}z^k
 	 & =  \sum_{k=n}^\infty \frac{b_k}{k!} \frac{m}{z^m} \int_0^{\rho e^{i\theta}} u^k e^{ -(u/z)^m } u^{m-1}\dint u \\
		& - \sum_{k=0}^{n-1}\frac{b_k}{k!} \frac{m}{z^m} \int_{\rho e^{i\theta}}^{\infty(\theta)}  u^k e^{ -(u/z)^m } u^{m-1}\dint u  
\end{aligned}
\]
logo
\[
\begin{aligned}
	\babs{ f^\rho(z) - \sum_{k=0}^{n-1}\frac{a_k}{k!}z^k }
		& \leq  \sum_{k=n}^{\infty}\frac{|b_k|}{k!} \frac{m}{|z|^m} \int_0^\rho r^k e^{ -\epsilon(r/|z|)^m } r^{m-1} \dint r \\
		& + \sum_{k=0}^{n-1}\frac{|b_k|}{k!} \frac{m}{|z|^m} \int_\rho ^\infty r^k e^{ -\epsilon(r/|z|)^m } r^{m-1}\dint r .
\end{aligned}
\]

Fazendo a mudança de variável \(r \mapsto \rho r\) e usando que, para \(r \in [0,1]\) e \(k = n, n+1, \ldots\), vale \(r^k \leq r^n\) e que, \(r \in [1,\infty)\) e \(k = 1, 2 , \ldots, n-1\), vale \(r^k \leq r^n\), obtemos

\[
\begin{aligned}
	\babs{ f^\rho(z) - \sum_{k=0}^{n-1}\frac{a_k}{k!}z^k }
		& \leq   \sum_{k=0}^{\infty}\frac{|b_k|}{k!}\rho^{k+m} \frac{m}{|z|^m} \int_0^\infty r^n e^{ -\epsilon(\rho r/|z|)^m } r^{m-1} \dint r \\
		& =  M (\epsilon^{-m}\rho^{-1}|z|)^n\Gamma(1+n/m),
\end{aligned}
\]

onde \(M = \sum_{k=0}^{\infty}\frac{|b_k|}{k!}\rho^{k+m} < \infty\) pois a série é convergente. 

\end{proof}

\begin{obs}
Na demonstração acima, a função \(f^\rho\) está bem definida para qualquer \(\rho > 0\) suficientemente pequeno. Como \(\fs{f}\) é Gevrey-\(s\) para \(s > 1\), vale que \(g\) é não constante.

Dessa forma, sem perda de generalidade, podemos assumir que \(\theta = 0\) e que existe um \(\rho > 0\) pequeno tal que \(\Re g(\rho) > 0\).

Seja \(\rho' > \rho\) tal que \(g(\rho')\) está definida e \(g(r)\) não se anula para \(r \in (\rho, \rho')\). Nessas condições, temos que \(f^\rho \neq f^{\rho'}\). De fato, basta notar que
\[
	\Re \bsbrac{ f^{\rho'}(1) - f^{\rho}(1) } = m\int_\rho^{\rho'} \Re \bsbrac{g(r)} e^{r^m}  \dint r > 0.
\]

Portanto, ainda não temos uma maneira canônica de ``somar'' uma série divergente.
\end{obs}

Exatamente como fizemos com a transformada de Laplace, queremos saber como a transformada de Borel se relaciona com a transformada de Borel formal.

\begin{teo}
\label{borelassint} Seja \(S = S(\theta, \alpha, \rho)\) um setor arbitrário e \(f\) holomorfa em \(S\). Suponha que \(\fs{f}(z) = \sum_{n=0}^\infty a_n z^n/n!\) é uma expansão assintótica de ordem \(s > 1\) de \(f\) em \(S\). Seja \(m > 1/2\) satisfazendo \(\alpha > \pi/m\) de forma com que \(\Bo{m} f\) seja holomorfa em \(\tilde{S} = S(\theta, \alpha - \pi/m)\). Definimos \(\tilde{s}\) pela regra
\[
\tilde{s} =
	\begin{cases}
		s - m^{-1}, & \text{ se } s - m^{-1} > 1; \\
		1, & \text{ se } s - m^{-1} \leq 1.
	\end{cases}
\]
Então \((\Bo{m}f) \gapprox{\tilde{s}} (\fBo{m}\fs{f}) \text{ em } \tilde{S}\).
\end{teo}

\begin{proof} Tomando \(\eta\) suficientemente próximo de \(\theta\), podemos supor que \(\gamma_m(\eta)\) está contida em um subsetor de \(S\). Assim, existem constantes positivas \(M\) e \(C\) tais que
\[
	\babs{ f(z) - \sum_{k=0}^{n-1} \frac{a_k}{k!} z^k } \leq MC^nn!^{s - 1}|z|^n, \quad z \in \gamma(\eta), \quad n = 0, 1, 2 \dots .
\]

Vamos dividir a curva \(\gamma_m(\eta)\) em três e estimar a integral da definição da transformada de Borel em cada uma delas.

Sejam \(\eta_1 = \eta - (\epsilon + \pi)/(2m),\ \eta_2 = \eta + (\epsilon + \pi)/(2m)\). Para \(\psi = \eta_1, \eta_2\), definimos
\(\gamma_{\psi}(t) \doteq te^{i\psi}\), para \(t \in (0,R)\).

Considere \(\gamma_R(t) = Re^{it}\), para \(t \in (\eta_1, \eta_2)\). Se \(u \in S(\eta,\epsilon/(2m))\), definimos

\begin{equation}
I_\gamma(u) \doteq \int_{\gamma} \bsbrac{ f(z) - \sum_{k=0}^{n-1} \frac{a_k}{k!} z^k } e^{ (u/z)^m } z^{-1} \dint z, \quad \gamma = \gamma_{\eta_1}, \gamma_{\eta_2}, \gamma_R.
\end{equation}

Assim,
\[
	\Bo{m} \bsbrac{ f(z) - \sum_{k=0}^{n-1} \frac{a_k}{k!} z^k }(u) = \frac{m}{2\pi i} \bpar{ I_{\eta_1} + I_{R} - I_{\eta_2} }(u).
\]

Vamos começar fazendo a estimativa no arco. 

Temos que
\[
	\begin{aligned}
		| I_R(u) | & \leq mR^nMC^nn!^{s - 1} \int_{\eta_1}^{\eta_2} e^ { \Re \bpar { u^m/(Re^{it})^m } } \dint t \\
			& \leq mR^nMC^nn!^{s - 1}(\eta_2 - \eta_1)e^{ (|u|/R)^m }.
	\end{aligned}
\]

Para \(u\) fixado e \(n\) grande, podemos escolher \(R = |u|(n/m)^{-1/m}\), \(\tilde{M} = mM(\eta_2 - \eta_1)\) e \(\tilde{C} = Ce^{1/m}\). Assim
\[
	| I_R(u) | \leq \tilde{M}\tilde{C}^nn!^{s - 1}|u|^n(n/m)^{-n/m}.
\]

Agora vamos fazer a estimativa nas partes radiais.

Precisamos fazer a conta apenas para \(\psi = \eta_1\), pois a estimativa é análoga para o caso \(\psi = \eta_2\).

Como \(\eta_1 = \eta - (\epsilon + \pi)/(2m)\) e
\(
	-\epsilon/(4m) < \arg u - \eta < \epsilon/(4m),
\)
temos que
\[
	\pi/2 + \epsilon/4 < m(\arg u - \eta_1) < \pi/2 + 3\epsilon/4
\]
e, diminuindo \(\epsilon\) se necessário, existe \(c > 0\) tal que, para \(x \in (\pi/2 + \epsilon/4, \pi/2 + 3\epsilon/4)\), \(\cos(x) < -c\).

Com isso,
\[
	\Re (u/z)^m = (|u|/|z|)^m\cos \bsbrac { k(\arg u - \arg z) } < -c((|u|/|z|)^m)
\]

e obtemos a seguinte estimativa:
\[
	\begin{aligned}
		| I_\psi(u) | & \leq MC^nn!^{s-1} \int_0^R t^{n-1} e^{ \Re \bpar{u^m/(te^{i\psi})^m } } \dint t \\
			& \leq MC^nn!^{s-1} \int_0^R t^{n-1} e^{ -c \bpar{|u|/t }^m } \dint t \\
			& \leq MC^nn!^{s-1} \int_0^R t^{n-1} e^{ -c \bpar{|u|/t }^m } \dint t.
	\end{aligned}
\]

Fazendo a mudança de variável \(y = c(|u|/|z|)^m\), temos que
\[
	\begin{aligned}
	\int_0^R t^{n-1} e^{ -c \bpar{|u|/t }^m } \dint t & = c^{n/m}|u|^mm^{-1}\int_{cn/m}^\infty y^{-n/m-1} e^{ -y } \dint y \\
	& \leq |u|^nc^{-1}(n/m)^{-n/m-1}m^{-1}\int_0^\infty e^{-y} \dint y.
	\end{aligned}
\]

A demonstração pode ser facilmente concluída usando a Fórmula de Stirling.


\end{proof}

Nos dois próximos teoremas provaremos a afirmação feita anteriormente de que a transformada de Borel e a transformada de Laplace são inversas uma da outra.

\begin{teo}
\label{formulainversao} Com as mesmas notações do Teorema \ref{borelassint}.

Seja
\[
	g(u) = (\Bo{m}f)(u), \quad u \in \tilde{S}.
\]
Então \(g\) tem crescimento exponencial de ordem \(m\) em \(\tilde{S}\) e \(\La_m g\) é holomorfa em \(\hat{S} = S(\theta, \hat{\alpha},\hat{\rho})\), com \( \pi/m < \hat{\alpha} < \alpha\) e \(0 < \hat{\rho} \leq \rho\). 

Além disso, vale que
\[
	f(z) = (\La_m g)(z), \quad z \in \hat{S}.
\]
\end{teo}

\begin{proof} Primeiro vamos verificar que \(g\) tem crescimento exponencial de ordem \(m\).

Dividindo a curva \(\gamma_m\) em três partes e fazendo as estimativas exatamente como fizemos na demonstração do teorema anterior, vemos claramente que a integral sobre as partes radiais tendem para zero uniformemente quando \(u \to \infty\) em \(S(\eta,\epsilon/(2m))\).

A integral sobre o arco claramente define uma função inteira que possui possui crescimento exponencial de ordem \(m\). De fato, a integral no arco é dada por
\[	
	g_1(u) = \frac{1}{2 \pi i} \int_{\gamma_R} f(z)e^{(u/z)^m}z^{-1} \dint z,
\] que claramente está definida para qualquer \(u \in \C\). Falta então verificar que essa função tem crescimento exponencial de ordem \(m\). Temos que

\[
	\begin{aligned}
		|g_1(u)| & \leq \frac{R}{2 \pi} \int_{\eta_1}^{\eta_2} |f(Re^{it})|e^{\Re (u/Re^{it})^m} \dint t \\
		& \leq e^{(|u|/R)^m} \frac{R}{2 \pi} \int_{\eta_1}^{\eta_2} |f(Re^{it})| \dint t.
	\end{aligned}
\]


Combinando essas informações, vemos que \(g(u) = (\Bo{m})f(u)\) tem crescimento exponencial de ordem \(m\) em \(S(\eta,\epsilon/(2m))\). Variando \(\eta\) concluímos que o mesmo vale para \(\tilde{S}\).

Vamos provar agora que \(f = \La_{m}g\).

Como ambas as funções são holomorfas em \(\hat{S}\), basta provar a igualdade para \(z\) com \(\arg z = \theta\) e \(|z|\) suficientemente pequeno.

Temos que a seguinte integral converge:
\[
	\La_m (\Bo{m} f)(z) = \frac{m^2}{2\pi i z^m} \int_0^{\infty(\theta)} \bpar{ \int_{\gamma_m(\theta)} f(w) e^{(u/w)^m}w^{-1} \dint w } e^{-(u/z)^m} u^{m-1}\dint u
\]

Para trocar a ordem de integração, precisamos verificar as hipóteses do Teorema de Fubini.

Observemos que para \(w\) nas regiões radiais vale que
\[
	|f(w)w^{-1}u^{m-1}||e^{(u/w)^m}||e^{-(u/z)^m}| \leq |f(w)w^{-1}u^{m-1}|e^{-c|u/w| - |u/z|},
\]
para algum \(c > 0\), já que \(\arg (u/z) = 0\) e \(m \arg (u/w) > (\pi + \epsilon)/2\), para \(\epsilon > 0\) pequeno.

Quando \(w\) está no arco, vale que \(\cos(m \arg(u/w)) \leq 1\) e \(|w| > |z|\).

Assim, temos
\[
	|f(w)w^{-1}u^{m-1}||e^{(u/w)^m}||e^{-(u/z)^m}| \leq |f(w)w^{-1}u^{m-1}|e^{|u/w| - |u/z|}.
\]

Como \(|z|\) é pequeno, podemos primeiro integrar em \(u\). Além disso, podemos integrar em \(w\), pois o caminho \(\gamma_\theta(\theta)\) é compacto. 

Logo, podemos trocar a ordem de integração, obtendo:

\[
	\begin{aligned}
		\La_m (\Bo{m} f)(z) 
			& = \frac{m^2}{2\pi i z^m} \int_{\gamma_m(\theta)} f(w) \bpar{ \int_0^{\infty(\theta)} e^{ u^m(w^{-m}-z^{-m}) }u^{m-1}\dint u } w^{-1} \dint w \\
			& = \frac{1}{2\pi i z^m} \int_{\gamma_m(\theta)} \frac{w^{m-1} f(w)}{w^m - z^m} \dint w.
	\end{aligned}
\]

A função \(F_z(w) = (w^{m-1} f(w))/(w^m - z^m)\) possui, no interior de \(\gamma_m(\theta)\), uma única singularidade dada em \(w = z\), que é um polo de ordem \(1\) e resíduo \(f(z)/m\). Portanto, basta aplicar o Teorema do Resíduo e a demonstração estará concluída.
\end{proof}

\begin{prp} Nas mesmas condições do Teorema \ref{laplaceassint}, a função \(\Bo{m}g\) é holomorfa em \(S\) e \(f = \Bo{m}g\) em \(S\).
\end{prp}

\begin{proof} Do Teorema \ref{laplaceassint} e do Teorema \ref{borelassint}, temos que \(\tilde{f} = \Bo{m}g\) é holomorfa em \(S\).

Usando o Teorema \ref{formulainversao}, sabemos que \(\La_{m}\tilde{f} = g = \La_m f\) em um setor bissectado por \(\theta\). E, usando a Proposição \ref{laplacekernel}, obtemos que \(f = \tilde{f}\).
\end{proof}

Finalmente temos condições necessárias e suficientes para que uma série formal seja somável. É interessante observar aqui como o uso de funções com crescimento exponencial é importante na caracterização.

\begin{teo}
\label{caracdasoma} Seja \(s\) satisfazendo \(1 < s < 3\) e \(\fs{f} \in \C_s[[z]]\). Então, escrevendo \(m = (s - 1)^{-1}\), segue que \(g(u) = (\fBo{m}{\fs{f}})(u)\) é holomorfa em uma vizinhança da origem. São equivalentes as seguintes afirmações:
	\begin{enumerate}
	
		\item \label{teoseilaitema} Existem um setor \(S = S(\theta,\alpha,\rho)\), com \(\alpha > (s-1)\pi\), e uma função holomorfa \(f \in \Ho_{(s)}(S)\) tal que \(f \gapprox{s} \fs{f}\) em \(S\).

		\item \label{teoseilaitemb} Existe um setor \(\tilde{S} = S(\theta, \epsilon)\) de raio infinito tal que \(g\) admite continuação holomorfa em \(\tilde{S}\)  e possui crescimento exponencial de ordem \(m\) nesse conjunto.
	\end{enumerate}
	Além disso, se qualquer um dos itens anteriores for verdadeiro, então \(f = \La_m g \) e portanto f é unicamente determinada por \(\fs{f}\). 
\end{teo}

\begin{proof} \(\eqref{teoseilaitema} \implies \eqref{teoseilaitemb}\): Suponha que valem as condições do item \ref{teoseilaitema}. Do Teorema \ref{borelassint}, sabemos que \(\Bo{m} f\) está definida e é holomorfa em \(S(\theta, \epsilon)\), onde \(\epsilon = \alpha - \pi/m)\).

Além disso, \(\Bo{m} f \gapprox{1} \fBo{m} f\) em \(S(\theta,\epsilon)\), isto é, \(g = \fBo{m} f\) é holomorfa em uma vizinhança da origem e admite continuação holomorfa \(S(\theta,\epsilon)\).

Usando o Teorema \ref{formulainversao}, sabemos que \(g\) tem crescimento exponencial de ordem \(m\) em \(S(\theta, \epsilon)\).

\(\eqref{teoseilaitemb} \implies \eqref{teoseilaitema}\):
Suponha que valem as condições do item \ref{teoseilaitemb}. Como \(g = \fBo{m} f\) em uma vizinhança da origem, temos que \(g \gapprox{1} \fBo{m} f\).
Usando o teorema \ref{laplaceassint}, temos que existem \(\alpha > (s-1)\pi\) e \(\rho > 0\) tais que \(\La_m g \gapprox{s} \fs{f} = \fLa_{m} \fBo{k} \fs{g} \) em \(S(\theta, \alpha, \rho)\). Finalmente, o teorema \ref{formulainversao} garante que \(f = \La_{m} g\).

Observemos que como \(g\) é holomorfa, a série \(\fBo{m}\fs{f}\) está unicamente determinada e, portanto, \(f\) também é unicamente determinada por \(\fs f\).

\end{proof}

O teorema anterior não apenas dá uma nova demonstração da injetividade da Aplicação de Taylor para setores suficientemente grandes, como também dá uma condição necessária e suficiente para que uma série seja somável. 

Se provarmos que existe uma série formal que é Gevrey-s, mas que sua transformada de Borel formal de ordem \((s - 1)^{-1}\) não admite continuação holomorfa em nenhum setor de raio infinito, então a série não pode ser somável. Construiremos uma série com essa característica no próximo exemplo.

\begin{exe}
Seja \(s\) satisfazendo \(1 < s < 3\) e tomemos
\[
	\fs{f}(z) = \sum_{k=0}^\infty \Gamma \bsbrac{ 1 + 2^k(s - 1) } z^{2^k}.
\]

Vale que \(\fs{f} \in \C_{(s)}[[z]]\) e, para \(m = (s - 1)^{-1}\), a função
\[
	g(z) = (\fBo{m} \fs{f})(z) = \sum_{k=0}^\infty z^{2^k}
\]
está definida e é holomorfa no disco unitário.

A função \(g\) claramente satisfaz a seguinte família de equações funcionais:
\[
	g(z) = \sum_{k=0}^{n-1} z^{2^k} + g(z^{2^n}), \quad n \geq 1.
\]
Tome \(z_{j,n} = e^{2\pi i j 2^{-n}}\), para \(j,n \geq 1\). Observemos que, calculando \(g\) em \(rz_{j,n}\) e usando a equação funcional acima, obtemos
\[
	g(rz_{j,n}) = \sum_{k=0}^{n-1} r^{2^k}e^{2\pi i j 2^{k-n}} + g(r^{2^n}e^{2\pi i j}).
\]

O que a expressão acima nos mostra é que, se existisse o limite
\[
	\lim _{r \to 1} g(rz_{j,n}),
\]
ele teria que ser igual a
\[
	\sum_{k=0}^{n-1} e^{2\pi i j 2^{k-n}} + g(1).
\]
No entanto, \(g\) não pode ser definida holomorficamente em \(1\). Assim, como a sequência \(\{z_{j,n} : j,n = 1, 2, ...\}\) é densa no círculo unitário, segue que \(g\) não pode ser holomorficamente continuada para fora do disco unitário.

Portanto, não vale o item \ref{teoseilaitemb} do teorema anterior para tal \(\fs{f}\) e, consequentemente, também não vale o item \ref{teoseilaitema}.
\end{exe}

Finalmente obtemos um dos teoremas mais importantes deste texto, que nos dá informações precisas sobre a injetividade e sobrejetividade da transformada de Taylor.

	\begin{teo}
	\label{cap-desassgev:tays} Seja \(s > 1\) e considere a aplicação de Taylor
\[
	\tays{s}:\Hos{s}(S) \to \Cs{s}[[z]]
\]
definida em \eqref{cap-dev:aplicacao_de_taylor}.
		\begin{enumerate} 
		  \item Se o setor \(S\) tiver abertura menor ou igual a \(\min\{(s-1)\pi, 2\pi\}\), então \(\tays{s}\) é sobrejetiva mas não é injetiva;
		  \item Se \(s < 3\) e o setor \(S\) tiver abertura maior que \((s-1)\pi\), então a aplicação \(\tays{s}\) não é sobrejetiva, mas é injetiva.
		\end{enumerate}
	\end{teo}


\section{Relação com séries convergentes}

Nesta seção estamos interessados em responder a seguinte questão: se aplicarmos as técnicas estudadas neste trabalho para somar séries divergentes a uma série convergente, então a soma que encontramos é a mesma que o limite da série?

A resposta é sim. Começaremos provando um resultado que nos permite saber sob quais condições podemos aplicar a teoria desenvolvida.

\begin{prp} Seja \(f(z) = \sum_{k=0}^\infty a_kz^k/k!\) uma função inteira. Então \(f\) tem crescimento exponencial de ordem 1 se, e somente se, a série \(\sum_{k=0}^\infty a_kz^k\) tem raio de convergência positivo e finito.
\end{prp}

\begin{proof}

\((\implies)\) Suponhamos que \(f\) tem crescimento exponencial de ordem 1. Então existem contantes positivas \(b\) e \(B\) tais que
\[
	|f(z)| \leq Be^{b|z|}, \quad z \in \C.
\]

Usando a Fórmula de Cauchy, para \(R > 0\), temos
\[
	a_k = \frac{k!}{2\pi i} \int_{|z| = R} \frac{f(z)}{z^{k+1}} \dint z, \quad k = 1, 2, \ldots ,
\]
portanto vale a desigualdade:
\[
	|a_k| \leq k!B\frac{e^{bR}}{R^k}.
\]

Para cara \(k\) natural e \(R > 0\), seja \(\lambda_k(R) = \frac{e^{bR}}{R^k}\). Usando Cálculo Diferencial, sabemos que o mínimo dessa função é atingido em \(R = k/b\) e obtemos a seguinte desigualdade:
\[
	|a_k| \leq k!B\frac{e^kb^k}{k!} \leq B(eb)^k.
\]

Logo \(\sum_{k=0}^\infty a_k z^k\) tem raio de convergência positivo e finito.

\((\impliedby)\) Suponhamos agora que \(\sum_{k=0}^\infty a_k z^k\) tem raio de convergência positivo e finito. Então existem constantes positivas \(M,~C > 0\) tais que
\[
	|a_k| \leq MC^k, \quad k = 1, 2, \ldots,
\]
portanto
\[
	|f(z)| = \babs{ \sum_{k=1}^\infty \frac{a_k}{k!}z^k } \leq M \sum_{k=0}^\infty \frac{(C|z|)^k}{k!} = Me^{C|z|}.
\]
Como queríamos.
\end{proof}

Seja \(\fs{f}(z) = \sum_{k=0}^\infty a_kz^k\) uma série com raio de convergência finito. Essa série define uma função holomorfa \(f\) em uma vizinhança da origem.

Como \(\Cs{s}[[z]] \subset \Cs{s'}[[z]]\), se \(s < s'\), então \(\fs{f}(z) \in \Cs{2}[[z]]\). Portanto podemos tratar a série \(\fs{f}\) como se fosse Gevrey-2.

Tomemos \(g(z) = \fBo{1}{\fs{f}}(z)\) a transformada de Borel de \(\fs f\). Pela proposição anterior, sabemos que \(g\) não só está definida em uma vizinhança da origem, como também tem crescimento exponencial de ordem \(1\).

Podemos então aplicar a transformada de Laplace em \(g\) em qualquer direção \(\theta \in [-\pi,\pi]\) e obter uma soma de Borel da série \(\tilde{f}\) definida em um setor bissectado por \(\theta\) com abertura maior que \(\pi\). 

Agora, usando a injetividade da aplicação de Taylor, temos que \(\tilde{f}\) é igual a \(f\) nesse setor. Isto é, a função \(\tilde{f}\) pode ser holomorficamente continuada para um aberto contendo a origem. Portanto obtemos uma soma clássica da série.

O resultado acima pode então ser enunciado mais precisamente da seguinte forma:

\begin{prp} Seja \(f(z) = \sum_{k=0}^\infty a_kz^k\) uma série com raio de convergência finito. Se definirmos \(g(z) = \sum_{k=0}^\infty a_kz^k = \fBo{1}{f}(z)\), então \(f\) é uma continuação holomorfa da transformada de Laplace de \(g\).
\end{prp}

Esse último resultado confirma que, em um certo sentido, a teoria que desenvolvemos até aqui é uma generalização natural do conceito de convergência.


%

\chapter{Equações Diferenciais Ordinárias}

Neste capítulo, faremos aplicações em Equações Diferenciais Ordinárias.

Primeiro, vamos ilustrar que a Análise Assintótica de Gevrey se aplica em EDO. Para isso, voltaremos à Equação de Euler e mostraremos que a soma da solução formal vista no exemplo \ref{exe:soma_de_euler} é uma solução.

Como a teoria foi desenvolvida usando Análise Complexa, para fazermos aplicações em Equações Diferencias Ordinárias definidas em domínios reais, conectaremos as duas teorias enunciando alguns resultados de Equações Diferencias com domínios no plano complexo.

Ao conectar a teoria de Equações Diferenciais com Análise Complexa, uma abordagem natural é procurar soluções usando séries formais, porém, devido a existência de pontos irregulares, conceito que definiremos a seguir, pode ser que a série não convirja.

Portanto, para estudar a convergência, introduziremos o Polígono de Newton. Com ele é possível calcular a classe de Gevrey de uma solução formal de uma equação diferencial, o que nos permite aplicar os conceitos vistos no capítulo anterior.

Comecemos relembrando da equação e da série de Euler.

\begin{exe}
\label{cap:edo:exe:euler}
 Vimos, no exemplo \ref{des:equacao_de_euler}, que a série formal
\[
	\fs{f}(z) = \sum_{k = 1}^\infty (-1)^{k-1}(k-1)! z^k
\]
é uma solução formal da Equação de Euler
\begin{equation}
\label{exe:equacao_de_euler}
	\left\{
		\begin{aligned}
			z^2f' + f & = z; \\
			f(0) & = 0.
		\end{aligned}
	\right.
\end{equation}

Além disso, no exemplo \ref{exe:soma_de_euler}, encontramos uma função \(f\) que possui \(\fs{f}\) como expansão assintótica de ordem \(2\) no setor \(\{z \in \C: \Im z > 0\}\).

A função \(f\) é dada pela seguinte expressão:
\[
	f(z) = \ddint{0}{\infty} { \frac{ 1 }{1 + t} e^{-t/z} }{t}, \quad \Re z > 0.
\]

Naturalmente surge a seguinte pergunta: a função \(f\) é uma solução da Equação de Euler?

Para verificarmos, derivamos sob o sinal de integração,
\[
	f'(z) = \frac{1}{z^2}\int_0^{\infty} \frac{t}{1 + t} e^{-t/z} \dint t,
\]
e substituímos na equação
\[
	f(z) + z^2f'(z) = \int_0^\infty e^{-t/z} \dint t = z.
\]
Portanto, \(f\) é, de fato, uma solução da equação diferencial \eqref{exe:equacao_de_euler} em \(\Re z > 0\). 
\end{exe}

Nesse exemplo, foi possível facilmente ver que a soma é uma solução formal, pois encontramos uma expressão explícita para ela. No entanto, isso nem sempre acontece, como veremos nas aplicações a seguir. Nesses casos, as condições para a injetividade da aplicação de Taylor serão fundamentais para determinar se a soma é uma solução.


\section{Equações Diferenciais Complexas}

Nesta seção, introduziremos alguns conceitos de Equações Diferenciais Ordinárias para fazer uma conexão com a teoria de Análise Complexa. Começamos apresentando alguns resultados cujas demonstrações podem ser encontradas no capítulo X do livro \textit{A Course of Modern Analysis} de \textit{Whittaker and Watson}, \cite{whittaker1996course}.

Consideraremos Equações Diferenciais Lineares de Segunda Ordem na forma

\begin{equation}
\label{cap:edo:equation}
	\del^2_z u + p(z)\del_z u + q(z) u = 0,
\end{equation}
com \(p\) e \(q\) funções holomorfas em um aberto conexo \(\Omega\), exceto em um número finito de polos.

\begin{dfn}
Um ponto de \(\Omega\) em que \(p\) e \(q\) são holomorfas é chamado de \textbf{ponto ordinário} da equação. Outros pontos de \(\Omega\) serão chamados de \textbf{pontos singulares}.
\end{dfn}

Se \(b \in \Omega\) é um ponto ordinário da equação, como \(p\) e \(q\) são holomorfas em \(b\), existe \(r = r(b) > 0\) tal que \(D_r(b) \subset \Omega\) consiste apenas de pontos regulares.

Podemos enunciar o primeiro resultado.

\begin{prp}
\label{cap:edo:prp:ordinaria}
Sejam \(b\) um ponto ordinário da equação \eqref{cap:edo:equation} e \(r > 0\)  tal que \(D_r(b) \subset \Omega\) possui apenas pontos ordinários. Nessas condições, dados \(z_0,z_1 \in \C\), existe uma única solução holomorfa em \(\D_r(b)\) tal que \(u(b) = z_0,~u'(b) = z_1\). 
\end{prp}


Sabendo que existe uma solução \(u\) de \eqref{cap:edo:equation} que é holomorfa em \(D_r(b)\) e que satisfaz \(u(b) = z_0\) e \(u'(b) = z_1\), com \(z_0\) e \(z_1\) números complexos dados arbitrariamente, a maneira mais simples de encontrar tal solução é assumir que \(u(z) = \sum_{k=0}^\infty a_k (z - b)^k\), substituir a série na equação diferencial e encontrar uma relação entre os termos \(a_k\) para \(k = 0, 1, 2, \ldots . \)

\begin{dfn}
Suponhamos que \(c \in \Omega\) é um ponto onde \(p\) ou \(q\) possuem polos e que a ordem desses polos é tal que as funções \((z - c)p,~(z - c)^2q\) são holomorfas em uma vizinhança aberta de \(c\). O ponto \(c\) será chamado de \textbf{ponto regular}. Os pontos em que \(p\) ou \(q\) não satisfazem essa condição serão chamado de \textbf{pontos irregulares}.
\end{dfn}

Se \(c\) é um ponto regular, podemos reescrever a equação \eqref{cap:edo:equation} da seguinte maneira:
\begin{equation}
\label{cap:edo:equation2}
	(z - c)^2 \del^2_z u + (z - c)P(z - c) \del_z u + Q(z - c) u = 0,
\end{equation}
em que \(P\) e \(Q\) são funções holomorfas em \(c\).
Assim, para encontrar uma solução, primeiro escrevemos a série formal
\[
\fs{u}(z) = (z - c)^\alpha\bbrac{ 1 + \sum_{n=1}^\infty a_n(z - c)^n}
\]
e, em seguida, substituímos na equação e determinamos as constantes \(\alpha, a_1, ..., a_n\).

Resumindo, basta derivar formalmente e substituir na equação. Sob algumas condições, conseguimos encontrar uma relação entre os termos \(a_k\). Com essa relação, conseguimos provar que esses termos crescem obedecendo as estimativas de Cauchy, e portando a série \(\fs{u}(z)/(z - c)^\alpha\) define uma função holomorfa em uma vizinhança aberta do ponto \(c\).

O procedimento gera uma solução e, em alguns casos, até duas soluções, mas a presença de polos nos coeficientes impede que em geral encontremos um sistema fundamental de soluções. De qualquer forma, temos o seguinte resultado:

\begin{prp}
\label{cap:edo:prp:regular} Se \(c\) é um ponto regular da equação diferencial \eqref{cap:edo:equation2}, então existem \(\alpha\) e uma solução \(u\) da equação tal que \(u(z)/(z - c)^\alpha\) é holomorfa em uma vizinhança aberta de \(c\).
\end{prp}



\section{Polígono de Newton}

Nos dois casos que verificamos acima, foi possível verificar que a equação possui soluções holomorfas em determinados abertos de \(\Omega\). Porém, se tentarmos o mesmo procedimento em pontos irregulares, podemos não encontrar uma série convergente.

Existe uma maneira simples de saber quando uma solução formal da equação \eqref{cap:edo:equation} é convergente e, caso não seja convergente, conseguimos garantir não só que solução formal é uma série de Gevrey, como também conseguimos saber quais são as possíveis ordens de Gevrey da série.

Embora nosso interesse seja principalmente estudar Equações Diferenciais Ordinárias Lineares de Segunda Ordem, vamos enunciar algumas ferramentas que são aplicáveis para equações diferenciais lineares em geral. 

Seja \(P\) um operador diferencial linear de ordem \(n\) com coeficientes em \(\C\{z\}\),
\[
	P = P(z,\del_z) = \sum_{j=0}^n c_j(z)\del_z^j,
\]
onde \(c_n \in \C\{z\}\) é não nulo. Para cada função holomorfa \(c_j\), definimos \(o_j\) como a ordem da função \(c_j\) em \(0\).

Seja \(\Qua\) o segundo quadrante do plano \(\R^2\). Definimos \(\Qua(x,y) = \Qua + (x,y)\).

\begin{dfn} O \textbf{polígono de Newton} de \(P\) em 0, denotado por \(N(P)\), é a envoltória convexa dos conjuntos \(\Qua(j, o_j -j)\) para os \(j\) tais que \(c_j \neq 0\) em \(\C\{z\}\). Mais precisamente,
\[
	N(P) \doteq ch\bpar{ \bigcup_{j=0,~c_j \neq 0}^n \Qua(j, o_j -j) }.
\]
Denotaremos por \(0 = k_0 < k_1 < \ldots < k_l\) as inclinações dos lados de \(N(P)\).
\end{dfn}

O teorema a seguir, que foi enunciado por J.P. Ramis em \cite{ramis1978devissage}, nos diz que o Polígono de Newton nos dá uma  ``medida de irregularidade'' do operador \(P\).

\begin{teo}
\label{cap:edp:teo:ramis} Seja \(\fs{u}(z) = \sum_{j=0}^\infty a_j z^j/j! \in \C[[z]]\) tal que \(P\fs{u} = g \in \C\{z\}\). Então existe um número racional \(s \geq 1\) tal que \(\fs{u}(z) \in \Cs{s}[[z]]\), mas \(\fs{u}(z) \notin \Cs{s'}[[z]]\) para nenhum \(s'< s\). O número \(s\) é da forma \(s = 1 + 1/k\) com \(k \in \{k_1, k_2, ..., k_l\}\), o conjunto das inclinações positivas do Polígono de Newton de \(P\).
\end{teo}

Seguimos a convenção de que \(1/\infty = 0\), nesse caso \(s = 1\) e a solução formal é convergente.

Uma demonstração desse teorema pode ser encontrada em \cite{ramis1984theoremes}. Esse teorema possui diversas generalizações que podem ser encontradas, junto de aplicações em Sistemas Dinâmicos, em \cite{ramis1993series}.

Vejamos alguns exemplos:

\begin{exe} Seja \(P\) o operador definido na equação \eqref{cap:edo:equation}. Supondo que \(0 \in \Omega\) e que \(0\) é um ponto ordinário desse operador, temos que as funções \(p\) e \(q\) são holomorfas em \(0\), portanto a ordem delas, denotadas respectivamente por \(o_p\) e \(o_q\), em \(0\), é maior ou igual a \(0\). Logo o polígono de Newton de \(P\) é a envoltória convexa do conjunto \(\Qua(0,o_q) \cup \Qua(1,o_p-1) \cup \Qua(2,-2)\), ilustrada na figura \ref{fig:eq1}.

\begin{figure}[H]
\centering
\includegraphics{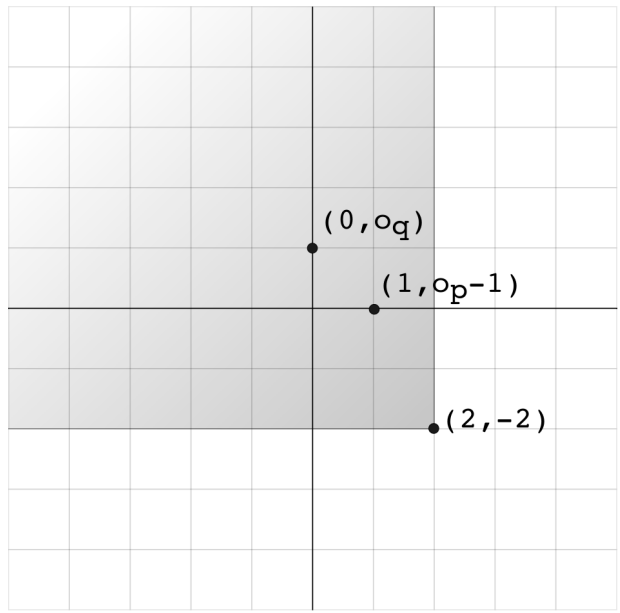}
\caption{Polígono de Newton do operador P, definido em \eqref{cap:edo:equation}. Nesse exemplo, \(o_q\) e \(o_p\) podem assumir qualquer inteiro não negativo. Para ilustrar, tomamos \(o_q = o_p = 1\). }
\label{fig:eq1}
\end{figure}

Como \(o_q \geq 0\) e \(o_p-1 \geq -1\), segue que \(N(P) = \Qua(2,-2)\) e a única inclinação positiva é \(k_1 = \infty\). Portanto, de acordo com o teorema anterior, as soluções formais desse sistema são convergentes em uma vizinhança da origem.
\end{exe}

\begin{exe} Supondo que \(0 \in \Omega\), consideremos agora \(P\) o operador definido em \ref{cap:edo:equation2} com \(c = 0\). Como \(P\) e \(Q\) são holomorfas em \(0\), supondo \(P \neq 0\), se denotarmos por \(o_P\) a ordem de \(zP(z)\) em \(0\) e por \(o_Q\) a ordem de \(Q\) em \(0\), temos que o polígono de Newton de \(P\) é a envoltória convexa do conjunto \(\Qua(0,o_Q) \cup \Qua(1,o_P - 1) \cup \Qua(2,0)
\), ilustrada na figura \ref{fig:eq2}.

\begin{figure}[H]
\centering
\includegraphics{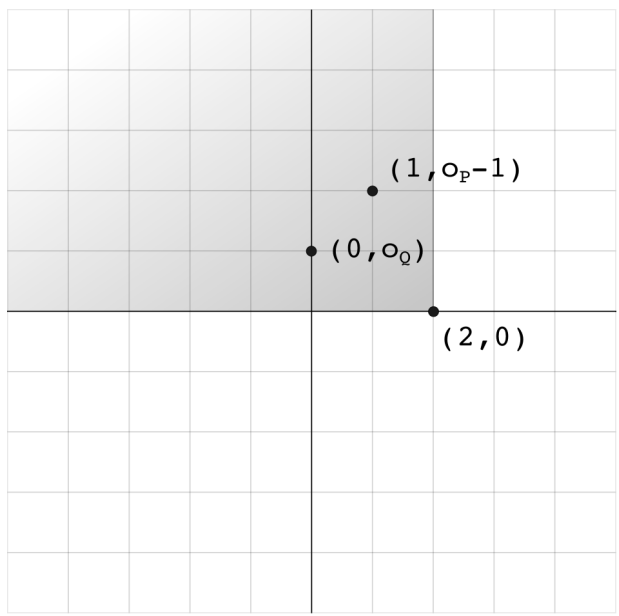}
\caption{Polígono de Newton do operador P, definido em \eqref{cap:edo:equation2}. Nesse exemplo, \(o_Q \geq 0\) e \(o_P \geq 1\). Para ilustrar, tomamos \(o_Q = 1\) e \(o_P = 3\).}
\label{fig:eq2}
\end{figure}

Como \(o_P \geq 1\) e \(o_Q \geq 0\), temos que \(N(P) = \Qua(2,0)\) e novamente uma solução formal desse sistema converge em uma vizinhança aberta da origem.
\end{exe}

 Os exemplos anteriores nos mostram que o Teorema \eqref{cap:edp:teo:ramis} é uma generalização da Proposição \eqref{cap:edo:prp:ordinaria} e da Proposição \eqref{cap:edo:prp:regular}.
 
 Motivados pela Equação de Euler, chamaremos de Operador de Euler o seguinte operador: \(P = z^2\del_z + 1\).
 
\begin{exe} O Polígono de Newton do operador de Euler é a envoltória convexa do conjunto \(\Qua(0,0) \cup \Qua(1,1)\), ilustrado na figura \ref{fig:eqeuler}.

\begin{figure}[H]
\centering
\includegraphics{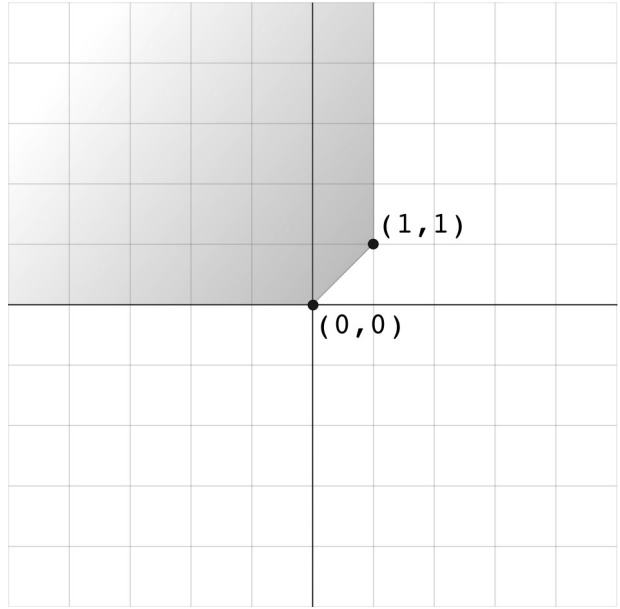}
\caption{Polígono de Newton do operador de Euler.}
\label{fig:eqeuler}
\end{figure}

Portanto, as soluções formais desse operador serão Gevrey-\(2\) ou Gevrey-\(1\).

No exemplo \ref{cap:edo:exe:euler}, vimos que a solução formal da equação diferencial \(P\fs{u} = z\) é Gevrey-\(2\). Facilmente conseguimos ver que a solução formal do problema \(P\fs{u} = 1\) é \(\fs{u} = 1\), que é claramente holomorfa, e portanto Gevrey-\(1\).
\end{exe}

Para demonstrar o Teorema \ref{cap:edp:teo:ramis}, é necessário usar ferramentas e técnicas que estão além dos objetivos deste trabalho.
No entanto, podemos provar que o teorema vale para o Operador de Euler. A demonstração pode ser feita usando apenas técnicas elementares.

\begin{prp}
\label{cap:edo:operadordeeuler} Seja \(P = z^2\del_z + 1\) o Operador de Euler. Suponha que \(\fs{f} \in \C[[z]]\) é uma série formal tal que \(P\fs{f} = g\), onde \(g \in \C\{z\}\). Então, ou \(\fs{f} \in \C\{z\}\), ou  \(\fs{f} \in \Cs{2}[[z]]\) e \(\fs{f} \notin \Cs{s}[[z]]\) para nenhum \(s < 2\).
\end{prp}

Antes de demonstrar esse resultado, vamos provar o seguinte lema que nos dá uma condição suficiente para que a classe de Gevrey de uma série formal seja ótima.

\begin{lem} Suponha que \(\fs{f}(z) = \sum_{k=0}^\infty a_kz^k \in \Cs{s}[[z]]\) seja tal que \(\sum_{k=0}^\infty a_kz^k/k!^{s-1}\) tem raio de convergência finito. Então não existe \(s' < s\) tal que \(\fs{f}(z) \in \Cs{s'}[[z]]\).
\end{lem}

\begin{proof} Por hipótese a série \(\sum_{k=0}^\infty a_kz^k/k!^{s-1}\) tem raio de convergência \(R \in (0,\infty)\). Sabemos que
\[
	\frac{1}{R} = \limsup_{k \to \infty} \sqrt[k]{ \frac{|a_k|}{k!^{s-1}} }.
\]

Se tivéssemos \(\fs{f}(z) \in \Cs{s'}[[z]]\), então a série \(\sum_{k=0}^\infty a_kz^k/k!^{s'-1}\) também seria convergente. Vamos mostrar que esse não é o caso.

Vamos calcular o raio de convergência. Temos
\[
	\sqrt[k]{ \frac{|a_k|}{k!^{s'-1}} } = \sqrt[k]{ \frac{|a_k|}{k!^{s-1}} k!^{s-s'} } = \sqrt[k]{ \frac{|a_k|}{k!^{s-1}}}\sqrt[k]{ k!^{s-s'} }.
\]
Como \(\limsup_{k \to \infty} \sqrt[k]{ k!^{s-s'} } = \infty\) e \(R\) é finito, segue que
\[
	\limsup_{k\to\infty}\sqrt[k]{ \frac{|a_k|}{k!^{s'-1}} } = \frac{1}{R}\infty = \infty.
\]

Provamos que o raio de convergência de \(\sum_{k=0}^\infty a_kz^k/k!^{s'-1}\) é 0 e portanto não vale \(\fs{f}(z) \in \Cs{s'}[[z]]\). Como queríamos.
\end{proof}

\begin{proof}[Demonstração da Proposição \ref{cap:edo:operadordeeuler}] Escrevendo \(\fs{f}(z) = \sum_{k=0}^\infty a_k z^k\), \(g(z) = \sum_{k=0}^\infty b_k z^k\) e procedendo como usualmente, facilmente conseguimos ver que a equação:
\[
	z^2\del_z \fs{f} + \fs{f} = \fs{g}
\]
é equivalente ao seguinte sistemas de equações:
\[
	\left\{\begin{aligned}
		a_0 &= b_0\\
		a_1 &= b_1\\
		a_1 + a_2 &= b_2\\
			&\ \ \!\vdots\\
		na_n + a_{n+1} & = b_{n+1}\\
			&\ \ \!\vdots\\
	\end{aligned}\right.
\]

Para \(n \geq 1\), multiplicamos a \(n\)-ésima linha do sistema acima por
\[
	\frac{(-1)^{n-1}}{(n-1)!}
\]
e obtemos outro sistema equivalente de equações:
\begin{equation}
	\left\{\begin{aligned}
		a_0 &= b_0\\
		-a_1 &= -b_1\\
		a_1 + a_2 &= b_2\\
			&\ \ \!\vdots\\
		\frac{(-1)^{n+1}a_{n+1}}{n!} & = - \bpar{ b_1 - b_2 + \ldots + \frac{(-1)^nb_{n+1}}{n!} } \\
			&\ \ \!\vdots\\
	\end{aligned}\right.
\end{equation}

Como \(g \in \C\{z\}\), existem constantes \(M,~C > 0\) tais que
\[
	\babs{ b_k } \leq MC^k, \qquad k = 1, 2, \ldots,
\]
portanto
\[
	\babs{ b_1 - b_2 + \ldots + \frac{(-1)^n b_{n+1}}{n!} } \leq MCe^C
\]
e claramente \(\fs{f} \in \Cs{2}[[z]].\)

Agora vamos tentar responder a seguinte pergunta: será que \(\fs{f} \in \Cs{s}[[z]]\) para algum \(s\) satisfazendo \(1 \leq s < 2\)?

Como \(\sum b_kz^k/k!\) é uma função inteira, podemos definir
\[
	\alpha(g) = b_1 - b_2 + \ldots + \frac{(-1)^n b_{n+1}}{n!} + \ldots.
\]

Precisamos tratar duas possibilidades: \(\alpha(g) \neq 0\) e \(\alpha(g) = 0\).

Suponhamos \(\alpha(g) \neq 0\). Nesse caso, o termo geral \(a_k\) da série \(\fs{f}\) não tende a zero quando \(k \to \infty\), portanto a série \(\sum_{k=0}^\infty a_k/k!\) não converge.

Assim concluímos que a série \(\sum_{k=0}^\infty a_kz^k/k!\) tem raio de convergência finito. Usando o lema anterior, concluímos que não existe \(s < 2\) tal que \(\sum_{k=0}^\infty a_kz^k \in \Cs{s}[[z]].\)

Suponhamos \(\alpha(g) = 0\). Nesse caso,
\[
	\begin{aligned}
	\frac{(-1)^n a_n}{(n-1)!} & = \frac{(-1)^n a_n}{(n-1)!} + \alpha(g)\\ & = \frac{(-1)^n b_{n+1}}{n!} + \frac{(-1)^{n+1} b_{n+2}}{(n+1)!} + \ldots
	\end{aligned}
\]
e obtemos a seguinte expressão para \(a_n\):
\[
	a_n = \frac{b_{n+1}}{n} - \frac{b_{n+2}}{n(n+1)} + \frac{b_{n+3}}{n(n+1)(n+2)} - \ldots.
\]

Como \(g \in \C\{z\}\), existem constantes positivas \(M,~C\) tais \(|b_n| \leq MC^n\) e portanto
\[
	\begin{aligned}
	|a_n| & \leq \sum_{k=1}^\infty \frac{|b_{n+k}|}{n(n+1) \ldots (n+k-1)} \\
		& \leq MC^n \sum_{k=0}^\infty \frac{C^k}{n(n+1)\ldots(n+k-1)}.
	\end{aligned}
\]

Do teste da razão, sabemos que a série
\[
	\sum_{k=1}^\infty \frac{C^k}{n(n+1)\ldots(n+k-1)} \doteq c_n, \quad n = 1, 2, \ldots
\]
converge.

Além disso, \(c_n \geq c_{n+1}\) e vale a desigualdade:
\[
	|a_n| \leq c_1MC^n, \quad n = 1, 2, \ldots.
\]
Assim \(\fs{f}\) é uma função holomorfa. 
\end{proof}

\section{Equações Diferenciais Reais}

Nesta seção, veremos como a Análise Assintótica pode ser usada para obter soluções de Equações Diferenciais em \(\R\).

Sejam \(p,~q\) e \(f\) funções em \(\C\{t\}\), com \(t\) real, e consideremos a seguinte equação:
\begin{equation}
	t^2\frac{\dint^2}{\dint t^2}u + p(t)\frac{\dint}{\dint t}u + q(t)u = f.
\end{equation}

Como \(p,~q\) e \(f\) são definidas por séries convergentes em uma vizinhança de \(0\), podemos passar o problema para uma variável complexa. Por simplicidade, representaremos por \(p,~q\) e \(f\) as extensões holomorfas dessas funções. Temos então uma equação na forma já considerada nas seções anteriores:
\[
	z^2 \del_z^2u + p\del_zu + qu = f.
\]

Se \(p\) e \(q\) se anularem em \(t = 0\) de ordem maior ou igual a \(2\), podemos considerar o caso homogêneo dessa equação e, usando a Proposição \ref{cap:edo:prp:ordinaria}, conseguimos um sistema fundamental de soluções.

Restringindo essas soluções a um intervalo de \(\R\), temos um sistema fundamental de soluções da equação diferencial real, portanto, basta usar o método da variação dos parâmetros para obter a solução que procuramos.

Mas, se estivermos em um caso mais geral, por exemplo, se, em \(z = 0\), \(p\) tiver um zero de ordem \(1\) e \(q\) tiver um zero de ordem \(0\), estamos nas condições da Proposição \ref{cap:edo:prp:regular}. Assim, eventualmente, podemos não conseguir um sistema fundamental do caso homogêneo, portanto pode ser que não encontremos uma solução holomorfa do caso não homogêneo.

Apesar disso, em alguns casos, conseguimos, usando Análise Assintótica de Gevrey, encontrar uma solução do caso real definida em um intervalo aberto que contém \(0\) e que é analítica em todo o intervalo, exceto em \(0\).

Antes de formalizarmos a afirmação acima, vamos provar o seguinte resultado:

\begin{teo} Seja \(P = P(z,\del_z)\) um operador diferencial linear com coeficientes holomorfos em uma vizinhança de \(0\). Suponhamos que \(\fs{u} \in \Cs{s}[[z]]\), com \(1 < s < 3\), satisfaz \(P\fs{u} = f \in \C\{z\}\). Se \(g = \fBo{(s-1)^{-1}}\fs{u}\) se estende holomorficamente e com crescimento exponencial de no máximo \((s-1)^{-1}\) a um setor \(S(\theta,\epsilon)\), para algum \(\theta \in [-\pi,\pi]\) e algum \(\epsilon > 0\), então existe uma função \(u \in \Hos{s}(S),~S = S(\theta,\alpha,\rho)\), \(\alpha > (s-1)\pi\) tal que \(u \gapprox{s} \fs{u}\) em \(S\) e vale que \(Pu = f\).
\end{teo}

\begin{proof} Como \(g\) se estende holomorficamente, com decrescimento exponencial de no máximo \((s-1)^{-1}\), ao setor \(S(\theta,\epsilon)\), o Teorema \ref{caracdasoma} garante que existe \(u \in \Hos{s}(S)\) satisfazendo \(u \gapprox{s} \fs{u}\). Além disso, a Proposição \ref{homo} garante que \(Pu \gapprox{s} P\fs{u}\).

Observemos que \(P\fs{u} = f\) em um disco de raio \(r > 0\), que denotamos por \(D_r(0)\). Isso é equivalente a dizer que \(f \gapprox{1} P\fs{u}\) em \(\tilde{S} = S \cap D_r\), logo também vale que \(f \gapprox{s} P\fs{u}\) em \(\tilde{S}\).

Portanto \(\tays{s}(Pu) = \tays{s}(f)\). Como a abertura de \(\tilde{S}\) é igual a abertura de \(S\), que é maior que \((s-1)\pi\), temos que a aplicação de Taylor é injetiva e assim \(Pu = f\) em \(\tilde{S}\). Logo \(Pu = f\) em \(S\). 
\end{proof}

Finalmente, podemos mostrar como a teoria que estudamos pode ser usada para obter resultados em Equações Diferenciais Reais. Vamos considerar um exemplo e depois enunciaremos o resultado.

\begin{exe} Considere agora o problema \(z^2\del_zu + iu = iz\). Com um procedimento análogo ao que fizemos para tratar a Equação de Euler, encontramos uma solução formal desse problema na forma
\[
	\fs{u}(z) = \sum_{k=1}^\infty i^{k-1}(k-1)!z^k
\]
e vemos que \(\fs{u}\) é expansão assintótica de ordem \(2\) da função
\[
	u(z) = \int_0^\infty \frac{1}{1 -it} e^{-t/z} \dint t.
\]

Essa função pode ser estendida holomorficamente ao aberto \(\Omega = \C \bs \{ri: r \leq 0 \}\). Usando o teorema anterior, sabemos que ela é solução da EDO em todo esse aberto.

Observemos que, dado um subsetor \(S \ssubset \Omega\) que contém \(\R\bs\{0\}\), do Lema \ref{cap:des:limorigem} sabemos que podemos estender continuamente u e todas as suas derivadas para 0. Portanto, \(u\) restrita a \(\R \bs \{0\}\) pode ser suavemente estendida a \(\R\) e essa extensão é uma solução em \(\R\) de:

\[
	\left\{
		\begin{aligned}
			t^2 \frac{\dint}{\dint t} u(t) + iu(t) & = ix; \\
			u(0) & = 0.
		\end{aligned}
	\right.
\]

\end{exe}

Agora podemos enunciar o seguinte corolário do teorema anterior.

\begin{cor} Seja \(P = P(t,\dint/\dint t)\) um operador diferencial  linear com coeficientes analíticos. Suponha que a versão complexificada do operador \(P = P(z,\del_z)\) esteja nas condições da teorema anterior. Suponha ainda que o setor em que a função \(g\) se prolonga seja \(S(\pi/2,\epsilon)\). Nessas condições, além de a função \(u\) satisfazer \(Pu = f\) em um aberto do plano complexo, a restrição de \(u\) a \(S(\pi/2,\alpha,\rho) \cap \R\) e todas as suas derivadas se estendem até \(0\). Essa restrição de \(u\) estendida até \(0\) é uma solução real de \(P(t,\dint/\dint t)\).
\end{cor}

\begin{proof} Resta apenas mostrar que a solução se estende suavemente até a origem. Para isso, basta aplicar o Lema \ref{cap:des:limorigem}. 
\end{proof}

\backmatter \singlespacing   
\bibliographystyle{alpha} 
\bibliography{bibliografia}  


\printindex   

\end{document}